\documentclass{article}

\usepackage{graphicx,hyperref,url}
\usepackage[all]{xy}
\usepackage[notref,notcite,final]{showkeys}

\newtheorem{theorem}{Theorem}
\newtheorem{lemma}[theorem]{Lemma}
\newtheorem{proposition}[theorem]{Proposition}
\newtheorem{corollary}[theorem]{Corollary}
\newtheorem{definition}[theorem]{Definition}

\renewcommand{\Im}{\operatorname{Im}}

\newcommand{\bull}{\text{\scalebox{.7}{$\bullet$}}}

\usepackage{amsfonts,amssymb,amsmath}
\newcommand{\C}{\ensuremath{\mathbb{C}}}

\newcommand{\Q}{\ensuremath{\mathbb{Q}}}
\newcommand{\R}{\ensuremath{\mathbb{R}}}
\newcommand{\T}{\ensuremath{\mathbb{T}}}
\newcommand{\Z}{\ensuremath{\mathbb{Z}}}

\date{}
\def\clock{{\count0=\time
           \divide\count0 60
           \ifnum\count0<10 0\fi\the\count0
           \multiply\count0 -60 \advance\count0 \time
           :\ifnum\count0<10 0\fi \the\count0
         }}
\newcommand{\timestamp}{{\small\vbox{\hbox{\tt\jobname.tex}
\hbox{\the\year/\the\month/\the\day, \clock}}}}

\newcommand{\footremember}[2]{%
\footnote{#2}
\newcounter{#1}
\setcounter{#1}{\value{footnote}}%
}
\newcommand{\footrecall}[1]{%
\footnotemark[\value{#1}]%
}

\begin{document}

\title{Elliptic fixed points with an invariant foliation:\\
Some facts and more questions}

\author{Alain Chenciner\footremember{IMCCE}{IMCCE (Observatoire de
    Paris, PSL Research University, CNRS)}\! \footnote{University
    Paris 7} ,
  David Sauzin\footrecall{IMCCE} \,\footremember{CNU}{Department of
    Mathematics, Capital Normal University, Beijing 100048, China} ,
  Shanzhong Sun\footrecall{CNU}\,
  \footnote{Academy for Multidisciplinary Studies, Capital Normal
    University, Beijing 100048, China} ,
  Qiaoling Wei\footrecall{CNU}}

\maketitle

\hskip5cm{\it dedicated to the memory of our friend} 

\hskip5cm{\it and colleague Alexey Borisov} 

\bigskip

\begin{abstract} We address the following question: let
  $F:(\R^2,0)\to(\R^2,0)$ be an analytic local diffeomorphism defined
  in the neighborhood of the non resonant elliptic fixed point 0 and
  let $\Phi$ be a formal conjugacy to a normal form $N$.  Supposing
  $F$ leaves invariant the foliation by circles centered at~$0$, what is
  the analytic nature of $\Phi$ and $N$ ?
\end{abstract}

\section{Motivation: the two families $A_{\lambda,a,d},B_{\lambda,a,d}$}\label{Ex1}

Understanding  the normalization of the following examples of local analytic diffeomorphisms of the plane with an elliptic fixed point was the motivation for raising the questions studied in this paper. Preserving the foliation by circles, these examples are radially trivial but angularly subtle; a normalization is a formal change of coordinates which makes the angular behavior trivial. 
\smallskip

\noindent $A_{\lambda,a,d}$ and $B_{\lambda,a,d}$ are the local maps from $(\R^2,0)$ to itself respectively defined by
\begin{equation*}
\left\{
\begin{split}
A_{\lambda,a,d}(z)&=\lambda z(1+a|z|^{2d})e^{\pi (z-\bar z)}\\
B_{\lambda,a,d}(z)&=\lambda z(1+a|z|^{2d})e^{\pi |z|^2(2i+z-\bar z)}\; ,
\end{split}
\right.
\end{equation*}
where $\lambda=\rho e^{2\pi i\omega},\; 0<\rho\le 1$ and $a\in\R,\, a<0$.
In polar coordinates $z=re^{2\pi i\theta}$~:

\begin{equation*}
\left\{
\begin{split}
A_{\lambda,a,d}(r,\theta)&=\bigl(\rho r(1+ar^{2d}),\theta+\omega+r\sin 2\pi\theta\bigr),\\
B_{\lambda,a,d}(r,\theta)&=\bigl(\rho r(1+ar^{2d}),\theta+\omega+r^2+r^3\sin 2\pi\theta\bigr).
\end{split}
\right.
\end{equation*}

\noindent We shall use the notations 
\begin{equation*}
\left\{
\begin{split}
A_{\lambda}(z)&=A_{\lambda,0,d}(z)=\lambda ze^{\pi (z-\bar z)},\\
B_\lambda(z)&=B_{\lambda,0,d}(z)=\lambda ze^{\pi |z|^2(2i+z-\bar z)}\,.
\end{split}
\right.
\end{equation*}

\noindent The families, parametrized by $r$, of analytic diffeomorphisms of the circle defined by the angular component of $A_{\lambda,a,d}$ and $B_{\lambda,a,d}$ are subfamilies of {\it Arnold's family} 
$$\theta\mapsto \theta+ s+t\sin2\pi\theta,$$
whose resonant zones (parameter values for which the rotation number is rational, the so-called {\it Arnold's tongues}) are depicted on figure 1.
In particular, each rational rotation number corresponds to an interval of values of $r$ (\cite{A,H}). 
\vskip0.5cm
\hskip -0.7cm
\includegraphics[scale=0.73]{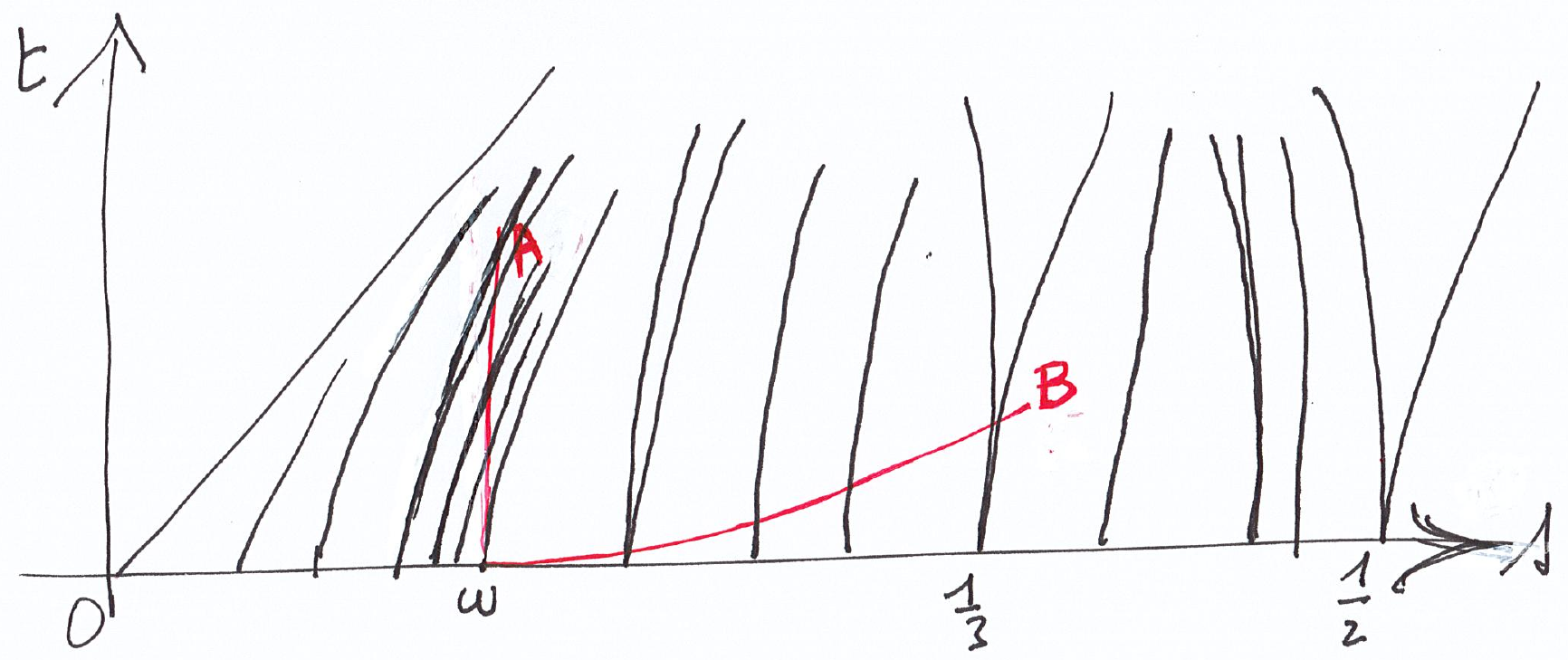}
\begin{center}
Figure 1 : Families $A$ and $B$.
\end{center}

\section{Formal theory}

\subsection{Special normal forms}\label{special}
\begin{definition}
Let ${\cal F}_0$ be the foliation of $\R^2$ by circles centered at 0. A formal diffeomorphism $F:(\R^2,0)\to(\R^2,0)$ is said to preserve ${\cal F}_0$ if $|F(z)|^2$ depends only on $|z|^2$. Identifying $\R^2$ with $\C$, this means that it is of the form 
\begin{equation*}
\left\{
\begin{split}
F(z)&=\lambda z\bigl(1+f(|z|^2)\bigr)e^{2\pi ig(z)}, \quad \hbox{where} \quad \lambda\not=0\in\C,\\
f(u)&=\sum_{n\ge 1}f_nu^n,\; f_n\in\R, \quad\quad g(z)=\sum_{j+k\ge 1}g_{jk}z^j\bar z^k,\; g_{kj}=\bar g_{jk}\in\C.
\end{split}
\right.
\end{equation*} 
\end{definition}
Blowing up the fixed point, that is using polar coordinates $z=re^{2\pi i\theta}$, turns the (formal) diffeomorphism $F$ into a skew-product
over the half-line $\R^+$ (which we shall still call $F$):
$$F:\R^+\times \T^1\to\R^+\times \T^1,\quad F(r,\theta)=\left(r\bigl(1+f(r^2)\bigr),\theta+\omega+g(r,\theta)\right).$$
Hence, iterating $F$ amounts to composing sequences of (formal) circle diffeomorphisms. 

\smallskip

\noindent The eigenvalues of the linear part $dF(0)$ of~$F$ are $\lambda$
and~$\bar\lambda$.
The case $|\lambda|<1$ is well understood since Poincar\'e: $F$ is
then locally formally conjugate to $dF(0)$ (and analytically if~$F$ is
analytic).
From now on, we shall suppose that $|\lambda|=1$ and even that
\[
 \lambda=e^{2\pi i\omega},\quad \omega\in\R\setminus\Q
\]
In that case, the only ``resonant monomials'', i.e.\ monomials
$z^p\bar z^q$ such that $\lambda^p\bar\lambda^q=1$, are those of the
form $|z|^{2p}$, since
\begin{equation}  \label{eqNRlambda} 
p\neq q \quad\Rightarrow\quad \lambda^p\bar\lambda^q-1\neq0.
\end{equation}

\begin{lemma}[Special normal forms]\label{pres}
  Let $F$ be a formal diffeomorphism of $\R^2$ defined in the
  neighborhood of the elliptic fixed point 0. Suppose that $F$
  preserves the foliation ${\cal F}_0$. If the derivative
  $dF(0)$ is a non-periodic rotation $z\mapsto \lambda z$, there exists
  a formal conjugacy $\Phi$ of $F$ to a normal form $N$ such that
  $\Phi$ preserves formally each circle centered at 0, hence $N$ sends
  formally each circle centered at 0 on the same circle as $F$ does,
  that is $\Phi\circ F = N \circ \Phi$ with
\begin{equation*}
\left\{
\begin{split}
\Phi(z)&=ze^{2\pi i\varphi(z)},\quad \varphi(z)=\sum_{p+q\ge 1}\varphi_{pq}z^p\bar z^q,\quad \varphi_{qp}=\bar\varphi_{pq},\\
N(z)&=\lambda z(1+f(|z|^2))e^{2\pi i n(|z|^2)},\quad n(|z|^2)=\sum_{s\ge 1}n_s|z|^{2s}.
\end{split}
\right.
\end{equation*}
Such conjugacies and the corresponding normal forms will be called
``special". The coefficients $\varphi_{pp}$ can be chosen arbitrarily in~$\R$. 
\end{lemma}

\begin{proof} 
Starting with 
$$F(z)=\lambda z(1+f(|z|^2))e^{2\pi ig(z)},\quad \quad\lambda=e^{2\pi i\omega},\;  \omega\notin\Q,$$
let us look for a formal change of coordinates 
$$\Phi(z)=ze^{2\pi i\varphi(z)},\quad \varphi(z)=\sum_{p+q\ge 1}\varphi_{pq}z^p\bar z^q,\quad \varphi_{qp}=\bar\varphi_{pq},$$
which transforms $F$ into 
a normal form 
$$N(z)=\lambda z(1+f(|z|^2))e^{2\pi i n(|z|^2)}\, .$$
The equation $\Phi\circ F=N\circ \Phi$ is equivalent to the homological equation
\begin{equation}
g(z)-n(|z|^2)+\varphi\circ F(z)-\varphi(z)=0.\tag{H}\label{Homoleq}
\end{equation}

\noindent Writing
\begin{multline*}
\bigl(1+f(|z|^{2})\bigr)e^{2\pi ig(z)}=
(1+\sum_rf_r|z|^{2r})\Big(1+\sum_s\frac{1}{s!}\big(2\pi i\sum_{t,u}g_{tu}z^t\bar z^u\big)^s\Big)
\\[1ex]
=1+\sum_{\alpha,\beta}c_{\alpha,\beta}z^\alpha\bar z^\beta,
\end{multline*}
the homological equation becomes
\begin{multline*}
-\sum_{j+k\ge 1}g_{jk}z^j\bar z^k+\sum_{s\ge 1}n_s|z|^{2s}\\[1ex]
=\sum_{p+q\ge 1}\varphi_{pq}\Big[\lambda^p\bar\lambda^q\bigl(1+\sum_{\alpha,\beta}c_{\alpha,\beta}z^\alpha\bar z^\beta\bigr)^p\bigl(1+\sum_{\alpha,\beta}\bar c_{\alpha,\beta}\bar z^\alpha z^\beta\bigr)^q-1\Big]z^p\bar z^q.
\end{multline*}
\noindent Once the coefficients $\varphi_{p'q'},\, p'+q'<p+q$ and $n_s,\, 2s<p+q$, are determined, identification of the terms in $z^p\bar z^q$ determines $\varphi_{pq}(\lambda^p\bar\lambda^q-1)$ if $p\not=q$ (resp. determines $n_{p}$ if $p=q$). 
In view of~\eqref{eqNRlambda}, this determines by induction the coefficients $\varphi_{pq}$ such that $p\not=q$. For $p=q$ the coefficient $\varphi_{pp}$ can be chosen arbitrarily provided it is real, hence the non unicity. 
\smallskip

\noindent Finally, the first member of the homological equation is real; replacing the equation by its conjugate and exchanging~$q$ and~$p$ amounts to the original equation apart from transforming 
$\varphi_{pq}$ into $\bar\varphi_{qp}$. 
This proves the lemma. 
\end{proof}

\begin{definition}\label{basic}
We shall call ``basic" and denote by $\Phi^*(z)$ the unique special
formal conjugacy without resonant terms in its angular component, i.e. 
$$\Phi^*(z)=ze^{2\pi i\varphi^*(z)},\quad \varphi^*(z)=\sum_{p+q\ge
  1,\ p\not=q}\varphi_{pq}z^p\bar z^q\; .$$ The corresponding normal
form $N^*=\Phi^*\circ F \circ (\Phi^*)^{(-1)}$ is called the basic normal form.
\end{definition}

\noindent All the other special formal conjugacies $\Phi(z)$ of $F$ to a normal form can be written
$$\Phi(z)=ze^{2\pi i\left(\varphi^*(z)+b(|z|^2\right)},$$
where $b$ is an arbitrary real formal series in one variable without constant term. A natural choice is given by the following lemma:
\begin{lemma}\label{n-pol}
  $1)$ If the valuation of $f$ is $d$, that is if $f(|z|^2)$ starts
  with a term in~$|z|^{2d}$, then the coefficients
  $n_s,\, 1\le s\le d,$ of a special normal form are uniquely
  determined; moreover, the $\varphi_{pp}$ can be chosen so that
  $n_{d+p}$ takes any given value, in particular 0, for $p\geq 1$.
If $f\equiv0$ (a case which we call conservative), then the special
normal form is uniquely determined.
\medskip
  
\noindent $2)$  The initial form $\tilde g_k$ of the formal function
$$\tilde g(z)=g(z)-\sum_{s\geq 1}n_s|z|^{2s}$$
does not contain any resonant term. In other words, 
$$\tilde g(re^{2\pi i\theta})=\tilde g_k(re^{2\pi i\theta})+O(r^{k+1}),\quad\hbox{\rm with}\quad \tilde g_k(re^{2\pi i\theta})=r^kP_k(\theta),$$
where $P_k(\theta)$ is a real trigonometric polynomial  of degree at most $k$ with mean value zero.
 
\end{lemma}

\noindent Hence, the stronger the contraction (i.e.\ the smaller is $d$), the less constrained is the torsion of a formal normal form (compare to section \ref{strong} were we recall that in case of a linear contraction (i.e.\ $|\lambda|<1$), the normal form can be chosen linear in the angle).
\smallskip

\begin{proof}
1) The homological equation \eqref{Homoleq} expresses $\tilde g$ as
\begin{equation*}
\begin{split}
\tilde g(z)&=\varphi(z)-\varphi\circ F(z)\\
&=\sum_{p+q\geq 1}\varphi_{pq}\left[1-\lambda^{p-q}\bigl(1+a|z|^{2d}+O(|z|^{2(d+1)})\bigr)^{p+q}e^{2\pi i(p-q)g(z)}\right]z^p\bar z^q.
\end{split}
\end{equation*}
The coefficient of $\varphi_{pp}$ is $|z|^{2p}\left[\bigl(1+f(|z|^2)\bigr)^{2p}-1\right]$ which starts with a term in ${|z|^{2(p+d)}}$: the choice of $\varphi_{11}$ allows choosing $n_{d+1}$, then the choice of $\varphi_{22}$ allows choosing $n_{d+2}$, and so on. 

\smallskip

\noindent 2) Let $k$ be the smallest degree of monomials in $\tilde{g}$, then
$$\tilde{g}_k(z)=\sum_{p+q=k}\varphi_{pq}\left[1-\lambda^{p-q}\right]z^p\bar
z^q$$  does not contain any resonant term.
\end{proof}
\medskip

\goodbreak

\subsubsection{The formal conjugacy equations for $F=A_{e^{2\pi i\omega},a,d}$}\label{FormalA}

\noindent We look for a conjugacy $\Phi(z)=ze^{2\pi i\varphi(z)}$ to a special normal form 
$$N(z)=\lambda z(1+a|z|^{2d})e^{2\pi in(|z|^2)},\quad \lambda=e^{2\pi i\omega},\quad \omega\;\hbox{irrational}.$$
The conjugacy equation $\Phi\circ F=N\circ\Phi$ becomes the homological equation
$$\frac{z-\bar z}{2i}+\varphi\bigl(\lambda z(1+a|z|^{2d})e^{\pi(z-\bar z)}\bigr)-\varphi(z)=n(|z|^2),$$
that is
$$\frac{z-\bar z}{2i}+\sum_{j+k\ge 1}\varphi_{jk}\left(\lambda^{j-k}(1+a|z|^{2d})^{j+k}e^{(j-k)\pi(z-\bar z)}-1\right)z^j\bar z^k=\sum_{s\ge 1}n_s|z|^{2s}.$$
Expanding the exponential we get
\begin{equation*}
\begin{split}
&\sum_{j+k\ge 1}\varphi_{jk}\left(\lambda^{j-k}(1+a|z|^{2d})^{j+k}\left(1+\sum_{n\ge 1}\frac{\pi^n}{n!}(j-k)^n(z-\bar z)^n\right)-1\right)z^j\bar z^k\\
&=-\frac{z-\bar z}{2i}+\sum_{s\ge 1}n_s|z|^{2s}.
\end{split}
\end{equation*}
Separating terms of degree 1 and 2 in $z,\bar z$ and the ones containing $a$ leads to
\[
\varphi_{10}=\frac{1}{2i(1-\lambda)}=\bar\varphi_{01},\quad
\varphi_{20}=\frac{\pi\lambda}{2i(1-\lambda)(1-\lambda^2)}=\bar\varphi_{02},\quad
n_1=-\Im\frac{\pi\lambda}{1-\lambda},
\]
and
\begin{multline*}
\sum_{j+k=N\ge 3}\varphi_{jk}(\lambda^{j-k}-1)z^j\bar z^k+\sum_{n\ge 1}\sum_{j+k=N-n}\varphi_{jk}\lambda^{j-k}\frac{\pi^n}{n!}(j-k)^n(z-\bar z)^nz^j\bar z^k\\[.5ex]
\hspace{-15em}  +\sum_{(j+k,l)\in I_N}\varphi_{jk}\lambda^{j-k}
a^l{j+k\choose l}|z|^{2dl}z^j\bar z^k\\[.5ex]
+\sum_{n\ge 1}\sum_{(j+k,l)\in I_{N-n}}\varphi_{jk}\lambda^{j-k}a^l{j+k\choose l}|z|^{2dl}\frac{\pi^n}{n!}
(j-k)^n(z-\bar z)^nz^j\bar
z^k=\left\{\begin{split}&n_{\frac{N}{2}}|z|^{N}\,\hbox{if $N$
      even}\\[.7ex] &0\,\hbox{if $N$ odd}\end{split}\right.
\end{multline*}
where the condition $(j+k,l)\in I_N$ means
$j+k=N-2dl\ge 1$ and $1\le l\le \frac{N}{2d+1}$.
\vspace{-6ex}

\begin{center}
\includegraphics[scale=0.28]{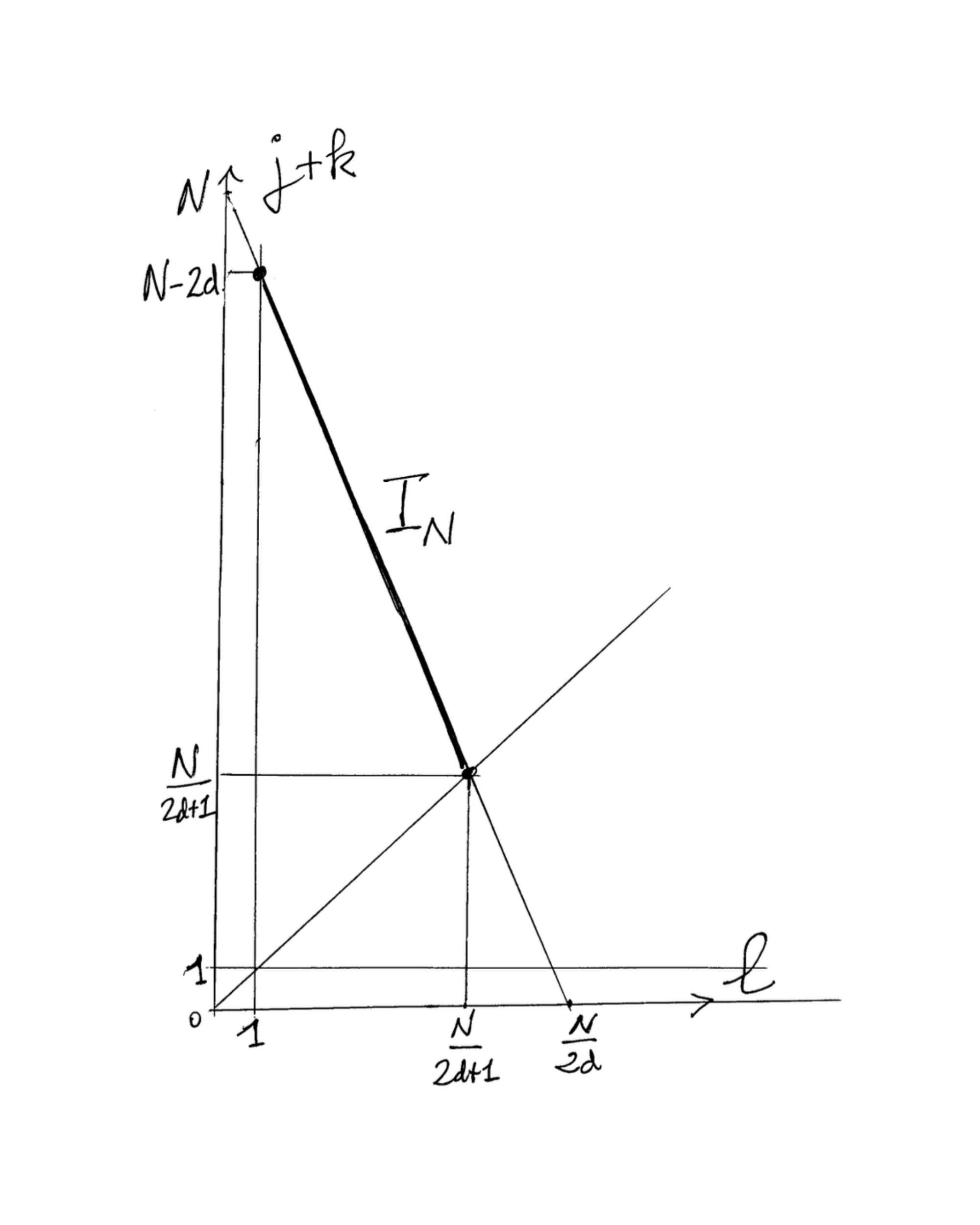}
\vspace{-3ex}

Figure 2: The  condition $(j+k,l)\in I_N$.
\end{center}

\noindent According to what we previously noticed, the coefficients
$\varphi_{pp}$ can be chosen so that
$n(|z|^2)=n_1|z|^2+\cdots+n_d|z|^{2d}$.
Each $n_s$, $1\le s \le d$, is a well-defined function of $(\lambda,
a, d)$, polynomial in~$a$ and rational in~$\lambda$.
\smallskip

\noindent Thus, if $F=A_{e^{2\pi i\omega},a,d}$,  
$$\tilde g(z)=\Im z-\frac{\pi}{2i} \, \frac{\lambda+1}{\lambda-1}|z|^2-n_2|z|^4-\cdots -n_d|z|^{2d}.$$

\subsubsection{The formal conjugacy equations for $F=B_{e^{2\pi i\omega},a,d}$}\label{FormalB}

\noindent The corresponding homological equation is
$$|z|^2\left(1+\frac{z-\bar z}{2i}\right)+\varphi\bigl(\lambda z(1+a|z|^{2d})e^{\pi|z|^2(2i+z-\bar z)}\bigr)-\varphi(z)=n(|z|^2),$$
that is
\begin{equation*}
\begin{split}
&\sum_{j+k\ge 1}\varphi_{jk}\left(\lambda^{j-k}(1+a|z|^{2d})^{j+k}e^{(j-k)\pi|z|^2(2i+z-\bar z)}-1\right)z^j\bar z^k\\
&=-|z|^2\left(1+\frac{z-\bar z}{2i}\right)+O(|z|^4).
\end{split}
\end{equation*}
Expanding the exponential we get
\begin{equation*}
\begin{split}
&\sum_{j+k\ge 1}\varphi_{jk}\left[\lambda^{j-k}(1+a|z|^{2d})^{j+k}\left(1+\sum_{n\ge 1}\frac{\pi^n}{n!}(j-k)^n|z|^{2n}(2i+z-\bar z)^n\right)-1\right]z^j\bar z^k\\
&=-|z|^2\left(1+\frac{z-\bar z}{2i}\right)+\sum_{s\ge 1}n_s|z|^{2s}.
\end{split}
\end{equation*}
Equating terms of degree up to 4 in $z,\bar z$ leads to
\begin{equation*}
\left\{
\begin{split}
&\varphi_{10}=\bar\varphi_{01}=0,\quad \varphi_{20}=\bar\varphi_{02}=0,\; n_1=1,\;
\varphi_{30}=\bar\varphi_{03}=0,\; \varphi_{21}=\bar\varphi_{12}=\frac{-1}{2i(\lambda-1)},\\
&\varphi_{40}=\bar\varphi_{04}=0,\; \varphi_{31}=\bar\varphi_{13}=0,\quad  n_2=0\;\hbox{if}\; d\ge 2 ;\quad  n_3=\Im \frac{\lambda}{\lambda-1}\;\hbox{if}\; d\ge 3.
\end{split}
\right.
\end{equation*}
\noindent From these computations, we deduce that, if $F=B_{e^{2\pi i\omega},a,d}$,  
$$\tilde g(z)=|z|^2\Im z+O(|z|^6).$$

\subsection{Non unicity of formal normal forms}\label{Nonu}
\noindent As we have already noticed in the special case, normal forms are not unique but, even in the general case, this non unicity is mild, more precisely:
\begin{lemma}[Non unicity of normal form]\label{Birk}
If $\lambda=e^{2\pi i\omega}$ with $\omega$ irrational, then any two
formal normal forms of the same formal diffeomorphism
$$N_1(z)=\lambda z\bigl(1+\sum_{k\ge 1}\alpha_k|z|^{2k})\quad\hbox{and}\quad N_2(z)=\lambda z\bigl(1+\sum_{k\ge 1}\beta_k|z|^{2k})\bigr)$$
are formally conjugated by a formal diffeomorphism of the form 
$$H(z)=z\bigl(1+h(|z|^2)\bigr)=z+\sum_{l\ge 1}h_l|z|^{2l}z,\quad h_l\in \C.$$
Moreover, 
the first non vanishing coefficient $\alpha_{k_0}$ of $N_1$ and the first non vanishing coefficient $\beta_{l_0}$ of $N_2$ coincide:
$$l_0=k_0\quad \text{and}\quad \beta_{k_0}=\alpha_{k_0}.$$
\end{lemma}
\begin {proof} Let $H(z)=z+\sum_{p+q\ge 2}h_{pq}z^p\bar z^q$ be such that $H\circ N_1=N_2\circ H$. Identifying degree 2 terms in this identity implies that
$(\lambda^p\bar\lambda^q-\lambda)h_{pq}=0$ for all $p$ and $q$ such that $p+q=2$. 
Moreover, degree 3 terms satisfy
$$\lambda \alpha_1|z|^2z+\sum_{p+q=3}\lambda^p\bar\lambda^qh_{pq}=\lambda\beta_1|z|^2z+\lambda\sum_{p+q=3}h_{pq}\, ,$$
from which it follows that $\alpha_1=\beta_1$ and $h_{pq}=0$ for all $(p,q)\not=(2,1)$ such that $p+q=3$.
Let us suppose by induction that
$$H(z)=z\bigl(1+\sum_{l=1}^mh_l|z|^{2l}\bigr)+\sum_{p+q\ge 2m+2}h_{pq}z^p\bar z^q.\eqno{(H_m)}$$
As all the terms of  $N_1(z)\left(1+\sum_{l=1}^mh_{l}|N_1(z)|^{2l}\right)$ are of odd degree, the only terms of degree $2m+2$ in $H\circ N_1(z)$ are $\sum_{p+q=2m+2}\lambda^p\bar\lambda^qh_{pq}z^p\bar z^q$.
\smallskip

\noindent Similarly, all the terms of $N_2\left(z\bigl(1+\sum_{l=1}^mh_l|z|^{2l}\bigr)\right)$ being of odd degree, the only terms of degree $2m+2$ in $N_2\circ H(z)$ are $\lambda\sum_{p+q=m+2}h_{pq}z^p\bar z^q$. 
\smallskip

\noindent One deduces that $h_{pq}=0$ for all $p$ and $q$ such that $p+q=2m+2$. Hence
$$H(z)=z\bigl(1+\sum_{l=1}^mh_l|z|^{2l}\bigr)+\sum_{p+q\ge 2m+3}h_{pq}z^p\bar z^q.$$
Terms of degree $2m+3$ in $H\circ N_1(z)$ are the ones of
$$N_1(z)\left(1+\sum_{l=1}^mh_l|N_1(z)|^{2l}\right)+\sum_{p+q=2m+3}\lambda^p\bar\lambda^qh_{pq}z^p\bar z^q$$
and those of $N_2\circ H(z)$ are the ones of
$$N_2\left(z\bigl(1+\sum_{l=1}^mh_l|z|^{2l}\bigr)\right)+\lambda\sum_{p+q=2m+3}h_{pq}z^p\bar z^q.$$
One deduces that
\begin{equation*}
\begin{split}
&\lambda\bigl(\alpha_{m+1}+\varphi(\alpha_1,\ldots,\alpha_m)\bigr)|z|^{2m+2}z+\sum_{p+q=2m+3}\lambda^p\bar\lambda^qh_{pq}z^p\bar z^q\\
=&
\lambda\bigl(\beta_{m+1}+\psi(\beta_1,\ldots,\beta_m)\bigr)|z|^{2m+2}z+\lambda\sum_{p+q=2m+3}h_{pq}z^p\bar z^q,
\end{split}
\end{equation*}
where $\varphi$ and $\psi$ are polynomials without constant term. It follows that $(H_{m+1})$ is verified and that $\alpha_{m+1}=\beta_{m+1}$ if all the $\alpha_k$ and the $\beta_k$ vanish for $k\le m$, which concludes the proof. 
\end {proof}

\begin{corollary}\label{any}
Under the hypotheses of lemma \ref{pres}, any formal conjugacy $\Psi$ of $F$ to a normal form preserves the foliation ${\cal F}_0$.
\end{corollary}

\noindent Indeed, writing 
\[H(z)=z\bigl(1+h(|z|^2)\bigr)=z\bigl(1+a(|z|^2)\bigr)e^{2\pi ib(|z|^2)},\]where $a$ and $b$ are real series, it follows from lemmas \ref{pres} and \ref{Birk} that, once the basic normal form 
$$N^*(z)=\Phi^*\circ F\circ {\Phi^*}^{-1}(z)=\lambda z(1+f(|z|^2))e^{2\pi i n^*(|z|^2)},\quad \Phi^*(z)=ze^{2\pi i\varphi^*(z)},$$ is known,  the most general conjugacy of $F$ to a normal form is a composition
$$\Psi(z)=H\circ\Phi^*(z)=z\bigl(1+a(|z|^2)\bigr)e^{2\pi i\bigl(\varphi^*(z)+b(|z|^2)\bigr)}$$
where $a(X)$ and $b(X)$ are arbitrary real formal series in one variable without constant term.
\smallskip

A direct computation leads to 
\begin{lemma}\label{general-nf}
Corresponding to a general formal conjugacy $\Psi=H\circ \Phi^*$ as above, the most general normal form for $F$ is
$$N(z)=\Psi\circ F\circ\Psi^{-1}(z)=\lambda z\bigl(1+\alpha(|z|^2)\bigr)e^{2\pi i\beta(|z|^2)},$$
where $\alpha$ and $\beta$ are given by the following formulas
\begin{equation*}
\begin{split}
1+\alpha(|z|^2)&=\bigl(1+f(|H^{-1}(z)|^2)\bigr)\frac{1+a\bigl(|F\circ H^{-1}(z)|^2\bigr)}{1+a(|H^{-1}(z)|^2)}\, ,\\
\beta(|z|^2)&=n^*(|H^{-1}(z)|^2)+b\bigl(|F\circ H^{-1}(z)|^2\bigr)-b(|H^{-1}(z)|^2),\; \hbox{with}\\
|H(u)|^2&=|u|^2\bigl(1+a(|u|^2)\bigr)^2,\;\hbox{hence}\; |H^{-1}(z)|^2=|z|^2\bigl(1+\rho(|z|^2)\bigr).
\end{split}
\end{equation*}
\end{lemma}
\smallskip

\noindent {\bf Remarks.}
\smallskip

\noindent {\bf 1)} If the conjugacies $\Phi_1$ and $\Phi_2$ are special, the composition $H=\Phi_2\circ\Phi_1^{-1}$ must preserve individually each circle: $H(z)=ze^{2\pi ib(|z|^2)}$. Hence the corresponding special normal forms $N_k(z)=\lambda z\bigl(1+f(|z|^2)\bigr)e^{2\pi in_k(|z|^2)},\, k=1,2$, satisfy
$$n_2(|z|^2)-n_1(|z|^2)=b\bigl(|F(z)|^2\bigr)-b(|z|^2).$$
\smallskip

\noindent {\bf 2)} If $f\equiv 0$, i.e.\ $|F(z)|=|z|$, a case which we
shall call {\it conservative}, then $\alpha\equiv 0$ and
$\beta(|z|^2) = n^*(|H^{-1}(z)|^2)$. This implies that
$\beta$ can be chosen to be a polynomial and even a monomial: indeed,
if $n^*(X)=n_pX^p+O(|X|^{p+1})$ with $n_p\neq0$, we can write
$\frac{n^*(r^2)}{n_pr^{2p}} = \big(1+a^*(r^2)\big)^{2p}$ with a
suitable real series~$a^*$ and, by choosing $a=a^*$ in~$H$ (with
any~$b$), we get 
$n^*(r^2) = n_p \big( r (1+a(r^2)) \big)^{2p}$,
whence $\beta(r^2)=n_pr^{2p}$.  \smallskip

\noindent Hence, {\it if $F$ is conservative, there always exist a non conservative formal transformation to a convergent normal form $N(z)=\lambda ze^{2\pi i|z|^{2p}}$.}

\smallskip

\noindent {\bf 3)} In general, even if~$f$ is not identically~$0$, we
can always achieve $\beta(r^2)=n_pr^{2p}$, i.e.\ an ``angularly
polynomial normal form'', by choosing $a=a^*$ as above and
$b=0$.
However, the resulting normal form $N(z)=\lambda z(1+\alpha(|z|^2))e^{2\pi
  i|z|^{2p}}$ is not convergent if~$n^*$ is not convergent.
\smallskip

\begin{definition}\label{RCconj}
A formal conjugacy $\Psi$ (resp. a normal form $\Psi\circ F\circ\Psi^{-1}$) such that $H$ (or what is equivalent, the series $a$ and $b$) converge will be called an RC (resonant part convergent) formal conjugacy (resp. RC normal form).
\end{definition}

\begin{lemma}\label{normalization}
The three properties: $F$ admits a convergent normalization, $F$
admits a convergent RC normalization, every RC-normalization of $F$ is
convergent, are equivalent.
\smallskip

\noindent In particular, if its basic normalization~$\Phi^*$ is divergent, then
all normalizations of~$F$ are divergent.
\end{lemma}
\begin{proof}
It follows from the observation that  the terms $z|z|^{2s}$ in $\Psi=H\circ\Phi^*$ originate only from $H$, the Cauchy-Hadamard formula for the radius of convergence of a formal series in several variables implies that the convergence of 
$\Psi$ implies the ones of $H$ and $\Phi^*$. 
\end{proof}

\medskip

\noindent Hence, one can restrict the discussion of convergence to RC-normal forms and even to the basic special one $N^*(z)$.  
Notice that it is not yet known whether a polynomial normal form is an RC-normal form but this seems unlikely.

\begin{definition}\label{presfol} Let $G:(\R^2,0)\to(\R^2,0)$ be a formal diffeomorphism; we shall say that the formal diffeomorphism $F$ preserves the ``formal foliation" ${\cal F}=G^{-1}({\cal F}_0)$ if  $G\circ F\circ G^{-1}$ preserves ${\cal F}_0$. 
\end{definition}

\begin{corollary}\label{one}  A formal diffeomorphism $F$ whose derivative $dF(0)$ is an irrational rotation cannot preserve more than one formal foliation.
\end{corollary}
\begin{proof} 
If $F$ preserves ${\cal F}_1=G_1^{-1}({\cal F}_0)$ and ${\cal F}_2=G_2^{-1}({\cal F}_0)$, and $\Psi$ is a formal conjugacy of $F$  to a normal form $N$, corollary \ref{any} implies that the formal diffeomorphisms $\Psi\circ G_1^{-1}$ and $\Psi\circ G_2^{-1}$ must both preserve ${\cal F}_0$. This means that $\Psi$ sends both ${\cal F}_1$ and ${\cal F}_2$ on ${\cal F}_0$, hence that 
${\cal F}_1={\cal F}_2=\Psi^{-1}({\cal F}_0).$
\end{proof}
\medskip

\noindent The case $F$ a pure homothety is the simplest counter-example to Corollary \ref{one} when the hypothesis on $dF(0)$ is not satisfied.
\goodbreak

\section{Polynomial normal forms}\label{pol}
\begin{proposition}\label{propol} As soon as $F$ is a weak contraction, there exists a formal conjugacy to a polynomial normal form $N_1=\lambda zP(|z|^2)$. 
\end{proposition}

\noindent Recall remark 2 above: if $F$ is conservative, there exists a non conservative conjugacy to a normal form $N(z)=\lambda ze^{2\pi i|z|^{2p}}$ ; on the other hand, if $|\lambda|<1$, Poincar\'e has proved that there exists an analytic conjugacy of $F$ to $N=dF(0)$, that is $N(z)=\lambda z$. Note that here, as in Poincar\'e's case, no preservation of a foliation is required.
\medskip

\begin{proof}
We look for a formal conjugacy of a normal form $$N_2(z)=\lambda z\nu_2(|z|^2),\quad \hbox{such that}\quad 
|\nu_2(z)|^2=1-b|z|^{2r}+O(|z|^{2(r+1)}),\quad b>0,$$ 
to a polynomial normal form $N_1(z)=\lambda z\nu_1(|z|^2)$. The main variable being $|z|^2$, the use of symplectic polar coordinates
$z=t^{\frac{1}{2}}e^{2\pi i\theta}$ is mandatory.  The two normal forms $N_i, i=1,2$, become
$$(t,\theta)\mapsto \left(F_i(t)=tf_i(t), \theta+\omega+g_i(t)\right),\quad\quad f_i(t)=\left|\nu_i(t)\right|^2.$$
According to section \ref{Nonu}, a conjugacy is necessarily of the form 
$$\Phi(z)=z\left(1+h(|z|^2)\right),$$
 that is 
$$(t,\theta)\mapsto \left(\phi(t)=t\varphi(t),\theta+\gamma(t)\right),\quad\quad \varphi(t)=\left|1+h(t)\right|^2,\quad \gamma(t)=\arg(1+h(t)).$$ 
The following diagram summarizes the situation. 

\medskip

\hskip 0.6cm
\includegraphics[scale=0.55]{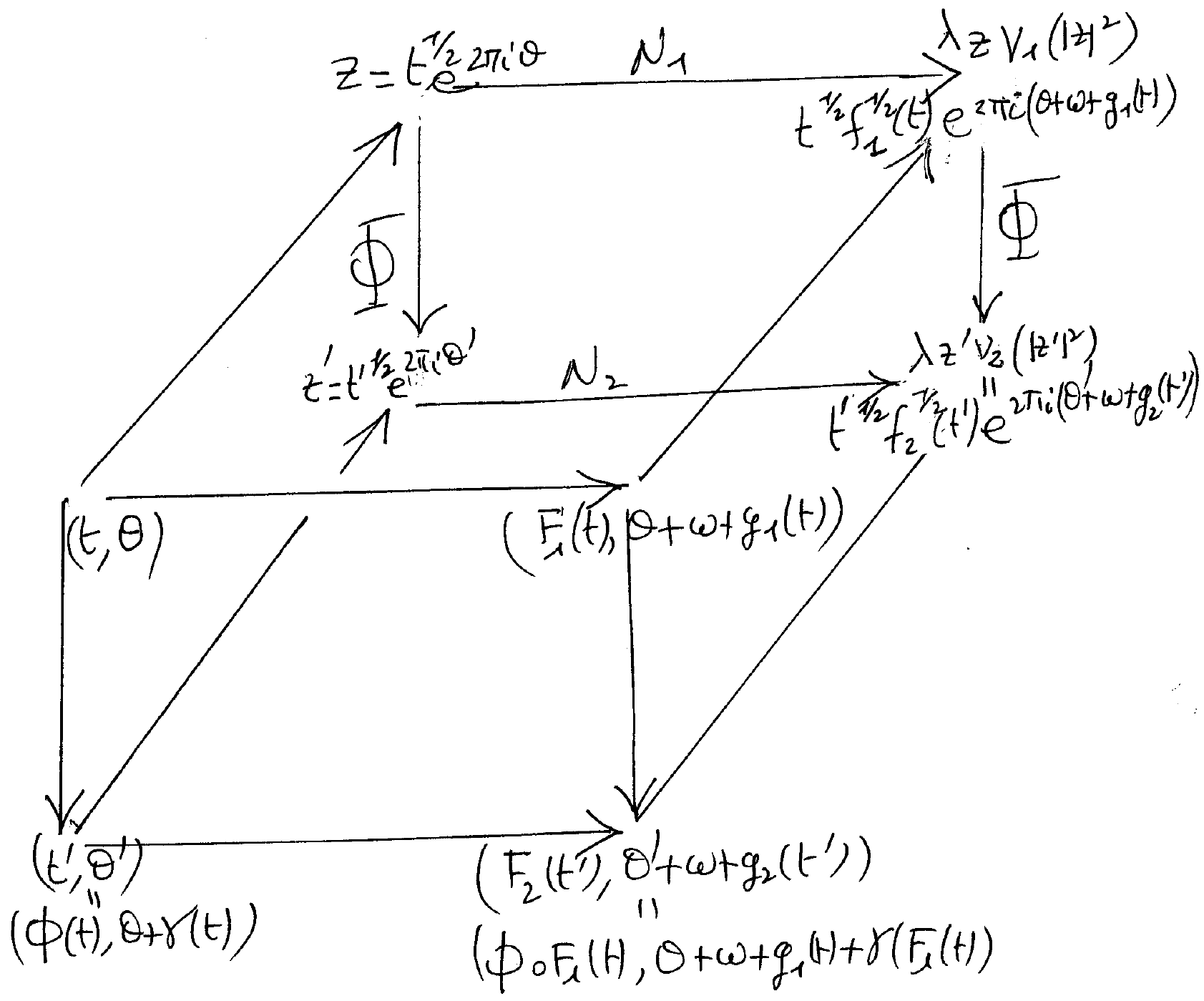}
\begin{center}
Figure 3 : Symplectic polar coordinates.
\end{center}

The conjugacy equation takes the form :
\begin{equation*}
\left\{
\begin{aligned}
&F_2\circ\phi=\phi\circ F_1, &\qquad\qquad\qquad &(CE1)\\[1ex]
&\gamma\circ F_1-\gamma+g_1=g_2\circ\phi. & &(CE2)
\end{aligned}
\right.
\end{equation*}

\noindent Equation $(CE1)$ is the conjugacy equation corresponding to the 1-dimensional real normalization problem. According to the formal analogue of \cite{Ch}, the map $F_2(t)=t-bt^{r+1}+O(t^{r+2})$ is formally conjugate to $F_1(t)=t-bt^{r+1}+ct^{2r+1}$, with $c\in\R$ uniquely defined by the $(2r+1)$-jet of $F_2$. Hence it is also conjugate to any map 
$F_1(t)=t-bt^{r+1}+ct^{2r+1}+O(t^{2r+2}),$ that is
\begin{lemma}
For any $F_1(t)\in t-bt^{r+1}+ct^{2r+1}+t^{2r+2}\R[[t]]$, there exists 
a formal diffeomorphism $\phi(t)\in t+t^2\R[[t]]$ such that such that $F_2\circ \phi=\phi\circ F_1$; moreover, the $(2r+1)$-jet of $\phi$ does not depend on the choice of such a $F_1$.
\end{lemma}
Being able to cope with the  $O(t^{2r+2})$ term is crucial. First, the following lemma will allow us to solve $(CE2)$ as soon as the $r$-jet of $g_1$ coincides with the $r$-jet  
$G(t)$ of $g_2\circ\phi$ {\it which, by the lemma, depends only on $N_2$ and not on the precise choice of the $O(t^{2r+2})$ terms in $F_1$}, that is not on $\phi$:
\begin{lemma}\label{homol} For any $F_1(t)\in
  t-bt^{r+1}+ct^{2r+1}+t^{2r+2}\R[[t]]$ with $b\neq0$, the linear operator $\gamma\mapsto \gamma\circ F_1-\gamma$ induces a bijection $t\R[[t]]\to t^{r+1}\R[[t]].$
\end{lemma}
\begin{proof}
Write $F_1(t)=t+tu(t)$ with $u(t)\in -bt^r+ct^{2r}+t^{2r+1}\R[[t]]$ and hence $u(t)\in -bt^r+t^{r+1}\R[[t]]$. Taylor formula yields
$$\gamma\circ F_1-\gamma= u\cdot \bigg(E\gamma+\sum_{k\ge 2}T_k\gamma \bigg),\quad E=t\frac{d}{dt},\; T_k=\frac{1}{k!}\, u^{k-1}t^k\Big(\frac{d}{dt}\Big)^k\cdot$$
The series of operators $\sum T_k$ is convergent in the following
sense: when applied to a formal series, $T_k$ increases its order by
at least $(k-1)r$ units (because~$u$ is of order~$r$). Now $E:t\R[[t]]\to t\R[[t]]$ is a bijection which does not change the order. It follows that $E+\sum_{k\ge 2}T_k$ is also a bijection whose inverse is defined by the convergent series of operators
$$\bigg(E+\sum_{k\ge 2}T_k\bigg)^{-1}=\bigg(Id+E^{-1}\sum_{k\ge 2}T_k\bigg)^{-1}\circ E^{-1}=\sum_{l\ge 0}(-1)^l\sum_{k\ge 2}\left(E^{-1}\circ T_k\right)^l\circ E^{-1}.$$
Finally, the conclusion follows from the fact that multiplication by 
$u=-bt^r+O(t^{r+1})$ is a bijection from $t\R[[t]]$ to $t^{r+1}\R[[t]]$ because $b\not=0$.
\end{proof}

\begin{corollary} There exists a polynomial $G(t)$ of degree $r$ such that, given any two formal series $\rho(t)$ and $\sigma(t)$, there exists a formal diffeomorphism $\phi\in t+t^2\R[[t]]$ and $\gamma\in t\R[[t]]$ which define a formal conjugacy of $N_2$ to 
$$N_1(t,\theta)=t^{\frac{1}{2}}\left[1-bt^{r}+ct^{2r}+t^{2r+1}\rho(t)\right]^{\frac{1}{2}}e^{2\pi i\left(\theta+\omega+G(t)+t^{r+1}\sigma(t)\right)},$$
that is
$$N_1(z)=\lambda z\left[1-b|z|^{2r}+c|z|^{4r}+|z|^{4r+2}\rho(|z|^2)\right]^{\frac{1}{2}}e^{2\pi i\left(G(|z|^2)+|z|^{2r+2}\sigma(|z|^2)\right)}.$$
\end{corollary}
\begin{proof}
One defines $G(t)$ as the $r$-jet of $g_2\circ\phi$ which, by the remark preceding lemma \ref{homol},  is independent of $\phi$.  Setting $g_1(t)=G(t)+t^{r+1}\sigma(t)$, one can solve $(CE1)$ and $(CE2)$.
\end{proof}
\medskip

\noindent {\bf Proof of Proposition \ref{propol}.}
It remains to notice that $\rho$ and $\sigma$ may be chosen so that $N_1$ be the polynomial $N_1(z)=\lambda zP(|z|^2)$, where $P(t)$ is the $2r$-jet of 
$\left[1-bt^r+ct^{2r}\right]^{\frac{1}{2}}e^{2\pi iG(r)}.$
\end{proof}
\smallskip

\noindent {\bf Remark.} In the same way, one can achieve a normal form
$N(z)=\lambda zQ(|z|^2)e^{2\pi iG(|z|^2)}$, where $Q(t)$ is the $2r$-jet of 
$\left[1-bt^r+ct^{2r}\right]^{\frac{1}{2}}$.

\section{Topological theory}\label{topth}
{\it In this section we do suppose that $F$ is not conservative and write 
$$F(z)=\lambda z(1+f(|z|^2))e^{2\pi ig(z)},\quad f(u)=au^d+O(u^{d+1}),\; d\ge 1,\;  a<0,$$
with $f:(\R,0)\to(\R,0)$ and $g:(\C,(0,0))\to(\R,0)$ real analytic germs.}
\smallskip

\noindent While in the formal theory Lemma \ref{n-pol} restrains the possible choices of $n$, for the topological theory any continuous $n:[0,R^2[\to\R$ vanishing at 0 is admissible and it is in this general setting that we recall
Sternberg's theorem (section \ref{Sternberg}).
\smallskip

\noindent Nevertheless, in section \ref {refinedTop} we come back to normal forms  $N$ satisfying the restrictions given in Lemma \ref{n-pol} when studying
the existence of a special type of topological semi-conjugacies of $F$ to $N$. 

\subsection{Sternberg's theorem}\label{Sternberg}
Here a normal form is a local homeomorphism of $(\R^2,0)$ which commutes with the group of rotations, that is which sends each small circle centered at 0 to another such circle by a rotation.
If the local contraction $F$ preserves the foliation ${\cal F}_0$, we call a normal form $N$ $F$-special if it sends each small circle $C$ onto the image $F(C)$ of this same circle by $F$.
 
\smallskip

\noindent One deduces from \cite{S} that any two local contractions are topologically conjugate one to the other in the neighborhood of  0. Moreover, if $F$ preserves the foliation ${\cal F}_0$, the proof in \cite{S} gives naturally a local topological conjugacy $\Phi$ to {\it any} $F$-special normal form $N$ such that, as in section \ref{special}, $\Phi$ preserves individually each circle. 
Indeed, let us choose any $F$-special normal form$N$; if $\cal D$ is a small disk  centered at 0 and ${\cal C}_0$ is its boundary, we may define $\Phi$ on ${\cal C}_0$ and its image ${\cal C}_1=F({\cal C}_0)$ by $$\Phi|_{{\cal C}_0}=Id\quad \hbox{and}\quad  \Phi|_{{\cal C}_1}=N\circ F^{-1}.$$
(${\cal C}_0$ and ${\cal C}_1$ are disjoint if $F$ is a contraction because $F$ preserves the foliation by circles, see figure 4). 

\hskip-0.5cm
\includegraphics[scale=0.6]{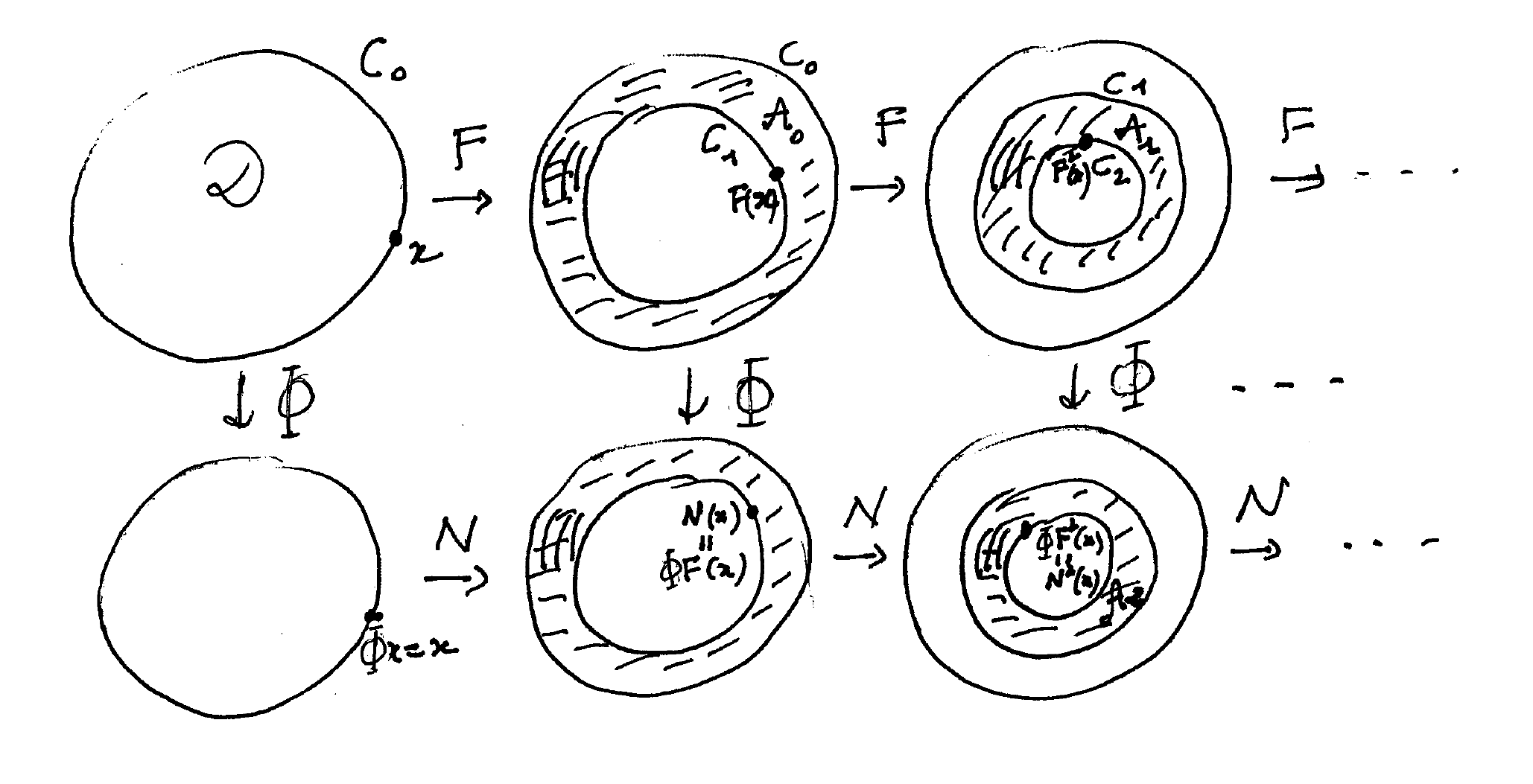}
\begin{center}
Figure 4 : Topological conjugacy.
\end{center}

\noindent It is then possible to extend $\Phi$ to the annulus ${\cal A}_0={\cal D}\setminus int F({\cal D})$ by an interpolation which preserves the foliation by circles centered at 0. Indeed, it is always possible to choose an analytic (if $N$ is chosen analytic) family of analytic diffeomorphisms of the circle which interpolates between the $Id$ and $N\circ F^{-1}|_{{\cal C}_1}$ after ${\cal C}_0$ and ${\cal C}_1$ have both been identified by radial projection with the standard circle. One then extends $\Phi$ to the images $F^n({\cal A}_0)=N^n({\cal A}_0)={\cal A}_n$ of ${\cal A}_0$ by $\Phi|_{{\cal A}_n}=N^n\circ \Phi|_{{\cal A}_0}\circ F^{-n}.$ 
This defines a local homeomorphism $\Phi$ of $(\R^2,0)$. 

\subsection{Special semi-normalizations, homological equation, Neumann series}\label{refinedTop}

\subsubsection{Topological semi-conjugacies and the homological equation}
\begin{definition}\label{semiconj}
Given maps $F$ and $N$, we say that $\Phi$ establishes a semi-conjugacy of $F$ to $N$ if $\Phi$ is surjective and $\Phi\circ F=N\circ\Phi$. We speak of topological semi-conjugacy when $\Phi$ and $N$ are continuous.
\end{definition}
We are interested in {\it special semi-normalizations} given by {\it special semi-conjugacies} $\Phi$, i.e.\ such that $|\Phi(z)|=|z|$, of $F$ to {special normal forms} $N$, i.e.\ such that $N(z)=\lambda z(1+f(|z|^2))e^{2\pi in(z|^2)}$.
We write $\Phi(z)=ze^{2\pi i\varphi(z)}$, meaning that $\varphi$ is a priori defined only on a punctured neighborhood $D^*_R=\{z\in\C,\, 0<|z|<R\}$ of 0 and that $\Phi(0)=0$. The map $\varphi$ can be chosen continuous on $D^*_R$ if $\Phi$ is continuous on $D_R=\{z\in\C,\, |z|<R\}$ and conversely, if $\varphi$ is continuous on $D^*_R$, the map $\Phi:D_R\to D_R$ is continuous and surjective. {\it 

\noindent From now on we are interested only in the case $\Phi$ is continuous.}
\smallskip

\noindent The semi-conjugacy equation $\Phi\circ F=N\circ \Phi$ is equivalent to the equation $$\varphi-\varphi\circ F=\tilde g \mod\Z, \;\hbox{where} \;\tilde g(z)=g(z)-n(|z|^2).$$
As $\varphi$ and $n$ are continuous, this is equivalent to 
$$\exists k\in\Z,\;  \varphi-\varphi\circ F=\tilde g+k.\eqno(HE)_k$$
Since $\tilde g$ is continuous on $D_R$ and vanishes at 0, we are led to single out a particular class of topological special semi-conjugacies:
\begin{definition}\label{angultame}
We say that $\Phi(z)=ze^{2\pi i\varphi(z)}$ is {$\theta$-tame}  or {\it angularly tame} if $\varphi$ extends to a continuous function $\varphi:D_R\to\R$.
\end{definition}
\begin{lemma}\label{HE_0}
$\Phi$ is angularly tame if and only if $\varphi$ is continuous at 0 and satisfies
$$\varphi-\varphi\circ F=\tilde g.\eqno(HE)_0$$
\end{lemma}
\begin{proof}
$\varphi$ and $F$ being continuous on $D_R$ with $F(0)=0$, the limit when $z$ tends to 0 of the left hand side of $(HE)_k$ is $\varphi(0)-\varphi(0)=0$ while the limit of the right hand side is $\tilde g(0)+k=k$.  
\end{proof}
\begin{definition}\label{Neumann} 
We call {\it Neumann series} the series of the form 
$\sum_{m=0}^\infty \tilde g\circ F^{(m)}.$
\end{definition}
\begin{lemma}\label{ptwisecv}
If $\Phi$ is an angularly tame semi-conjugacy of $F$ to $N$, the Neumann series is pointwise convergent and $\varphi=\varphi(0)+\sum_{m=0}^\infty \tilde g\circ F^{(m)}$. In particular, there is at most one angularly tame semi-conjugacy  of $F$ to a given $N$ up to a rotation.
\end{lemma}
\begin{proof}
From $(HE)_0$ one gets, for each integer $M\ge 1$,
$$\varphi(z)=\sum_{m=0}^{M-1}\tilde g\circ F^{(m)}(z)+\varphi\circ F^{(M)}(z).$$
As $\varphi\circ F^{(M)}$ tends pointwise to 0 when $M$ tends to infinity, we get that $\varphi-\varphi(0)$ is the pointwise limit of the series $\sum_{m=0}^\infty \tilde g\circ F^{(m)}$. 
\end{proof}
\smallskip

\noindent Conversely, if the Neumann series is pointwise convergent, its sum $\varphi$ provides a solution to $(HE)_0$ and thus a semi-conjugacy $\Phi$ but a stronger property is needed to ensure continuity:
\begin{lemma}\label{uniformcv}
If the Neumann series $\varphi=\sum_{m=0}^\infty \tilde{g}\circ F^{(m)}$ is uniformly convergent on $D_R$, it defines an angularly tame semi-conjugacy of $F$ to $N$.
\end{lemma}
Note that injectivity is not granted.
\subsubsection{Existence of an angularly tame semi-conjugacy}

\begin{proposition}\label{tsc}
Let $F(z)=\lambda z(1+f(|z|^2))e^{2\pi ig(z)}$ be such that $f$ has
valuation~$d$ and let
$n^*(|z|^2) = \sum_{s=1}^d n_s|z|^{2s}$ be the polynomial determined by the
formal theory (see Lemma~\ref{n-pol}).
Let $N(z)=\lambda z(1+f(|z|^2))e^{2\pi in(|z|^2)}$ be a normal form.
\smallskip

\noindent 1) If $n(|z|^2)=n^*(|z|^2)+O(|z|^{2d+1})$,
then there exists an angularly tame semi-conjugacy of~$F$ to~$N$.
\smallskip

\noindent 2) If $n$ is analytic but not of this form---that is if $N$
is not a normal form in the sense of the formal theory---, then a topological
conjugacy exists thanks to Sternberg but no angularly tame
semi-conjugacy exists.
\end{proposition}

\begin{proof}
 Without loss of generality we can assume that
$$F(z)=\lambda z(1+f(|z|^2))e^{2\pi i(n^*(|z|^2)+\check g(z))},\quad\hbox{\rm with}\quad \check g(z)=O(|z|^{2d+1}).$$
Indeed, we can always perform a preliminary change of coordinate
  $z\mapsto z e^{2\pi i\varphi(z)}$ where~$\varphi(z)$ is an
  appropriate \emph{polynomial} function of $(z,\bar z)$ so that~$F$
  is normalized up to any arbitrary order.

\smallskip
  
\noindent In order to be able to apply lemmas~\ref{ptwisecv}
and~\ref{uniformcv}, we shall estimate the size of the general term of
the series $\sum_{m=0}^\infty \tilde{g}\circ F^{(m)}$ with
\[ \tilde g(z) = n^*(|z|^2)-n(|z|^2)+\check g(z). \]
\noindent The main step is controling the decrease of $|F^{(m)}(z)|$:
as $F$ preserves the foliation ${\cal F}_0$ by circles, the norm $|F^{(m)}(z)|$ of any iterate depends only on $r=|z|$ :
$$|F(z)|=\nu(r):= r(1+f(r^2))\quad\hbox{and}\quad |F^{(m)}(z)|=\nu^{(m)}(r).$$
\begin{lemma}\label{Contraction}
There exist $r_0,C,D,K$ such that, for all $r\in ]0,r_0[$ and $m\ge 1$, 
\begin{equation*}
\begin{split}
&(i)\quad 0<\nu^{(m)}(r)<r, \\
&(ii)\quad  \nu^{(m)}(r)\le Cm^{-\frac{1}{2d}},\\
&(iii)\quad  m\ge Kr^{-2d}\Rightarrow \, \nu^{(m)}(r)\ge Dm^{-1/2d}.
\end{split}
\end{equation*}
\end{lemma}

\begin{proof}
\noindent The inversion $${\cal I}:r\mapsto U=r^{-p}$$ exchanging 0 and infinity, will allow us to estimate $\nu^{(m)}(r)$ by comparing the transform of $\nu(r)$ to a translation.
\smallskip

\noindent The positive integer $p$ and the positive real number $\tilde a$  being fixed, let
$$\nu_{p,\tilde a}(r)=r\left(1+\tilde ar^p\right)^{-\frac{1}{p}}\;.$$
One checks immediately that 
$${\cal I}\circ \nu_{p,\tilde a}\circ {\cal I}^{-1}(U)=U+\tilde a,$$
which implies that
$$\nu_{p,\tilde a}^{(m)}(r)=r(1+m\tilde ar^p)^{-\frac{1}{p}}.$$
Now, as $a<0$,\; if $0< a_-<2d |a|<a_+,\quad \exists r_0>0\; \hbox{such that for} \; 0<r<r_0$
$$\nu_{2d,a_+}(r)\le\nu(r)=r+ar^{2d+1}+O(r^{2d+2})\le \nu_{2d,a_-}(r)\, .$$
As $\nu_{a_-}$ and~$\nu$ are increasing functions that preserve
$]0,r_0[$, this implies that there exists $r_0>0$ such that
\[ \nu^{(m)}_{2d,a_-}(r)\le\nu^{(m)}(r)\le \nu^{(m)}_{2d,a_+}(r)\quad
 \text{for all} \; r\in\,]0,r_0[ \; \text{and}\; m\ge1.
\]
The explicit formula for $\nu_{2d,\tilde a}^{(m)}(r)$ concludes the proof. 
\end{proof}
\medskip

\noindent For the first part of proposition~\ref{tsc}, we have $\tilde
g = \check g + O(|z|^{2d+1}) = O(|z|^{2d+1})$, hence $|\check g(z)|\le A|z|^{2d+1}$ on
some disc near~$0$ and
$$\left|\check g\circ F^{(m)}(z)\right|\le ACm^{-\frac{2d+1}{2d}},$$
which entails the uniform convergence of the series $\sum_{m=0}^\infty
\left|\check{g}\circ F^{(m)}\right|$. We can thus conclude by lemma~\ref{uniformcv}.
\medskip

\noindent For the second part of proposition~\ref{tsc},
we have 
$$\tilde g(z)=\gamma |z|^{2k}+O(|z|^{2k+1}),\quad \text{\rm with}\;\;
k\le d\;\; \text{and} \; \; \gamma\neq 0.$$
In particular, choosing $\tilde\gamma = \frac1 2|\gamma|$,
\[
 \text{$\tilde g(z)$ does not change
    sign and $|\tilde g(z)| \ge \tilde\gamma |z|^{2k}$ for $|z|$ small enough.}
\]
Since $|F^{(m)}(z)| = \nu^{(m)}(|z|)$, it follows that
the Neumann series $\sum_{m=0}^\infty \tilde g\circ F^{(m)}(z)$
and the series $\sum_{m=0}^\infty \bigl(\nu^{(m)}(|z|)\bigr)^{2k}$ are
of the same nature for each $z$ close enough to~$0$.
\smallskip

\noindent By Lemma \ref{Contraction}{\it (iii)}, $m\ge Kr^{-2d}$ implies that $\bigl(\nu^{(m)}(|z|)\bigr)^{2k}\ge Dm^{-k/d}$.
As $k/d\le 1$, one concludes to the divergence of both series, which,
according to lemma~\ref{ptwisecv},  prevents the existence of an angularly tame semi-conjugacy.
\end{proof}

\section{Analytical theory} 
\subsection{The conservative case}\label{ConsCase}

Recall (section \ref{Nonu}) that this means that $F(z)=\lambda ze^{2\pi ig(z)}$ preserves individually each circle centered at 0. 
\smallskip

\noindent As any formal conjugacy $\Psi$ of $F$ to a normal form $N$ preserves the foliation ${\cal F}_0$ (section \ref{Nonu}, Corollary \ref{any}), if $N$ is convergent, $\Psi$ will be divergent as soon as there exist arbitrary small radii $r$ such that the restriction $F_r$ of $F$ to the circle ${|z|=r}$ is not analytically conjugate to a rotation, in particular for $F=A_{\lambda,0,d}$ and $F=B_{\lambda,0,d}$. Indeed  if convergent such a conjugacy would provide conjugacies $\Psi_r$ between $F_r$ and $N_{r'}$ for all small enough radii $r$ (figure 5).

\hskip-0.5cm
\includegraphics[scale=0.77]{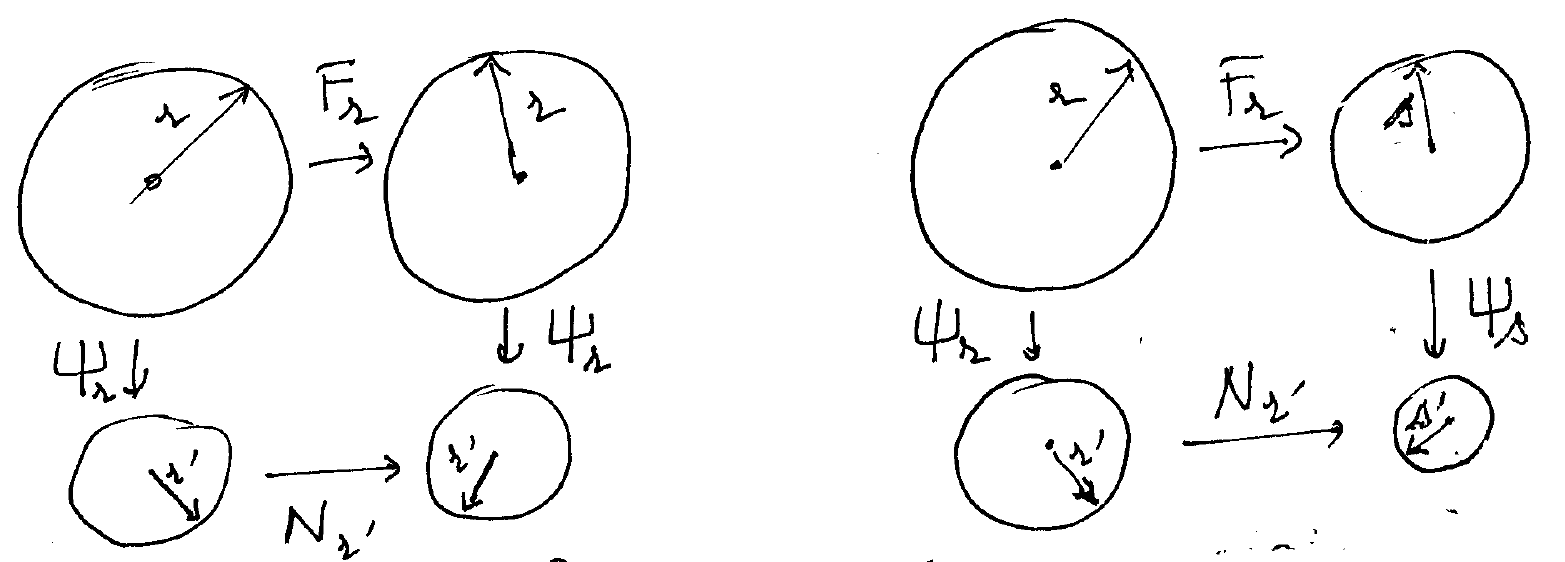}
\begin{center}
Figure 5: Conservative case.\hskip3cm Figure 6: dissipative case.
\end{center}

\subsection{Divergence implied by the holomorphic part $F^0$ of $F$}
Contrarily to what happens in the conservative case, the existence of an analytic conjugacy  $\Psi$ of $F$ to a normal form
$N=\Psi\circ F\circ \Psi^{-1}$ only implies, for each $r>0$, an identity $N_r=\Psi_{s}\circ F_r\circ \Psi_r^{-1}$ when restricting $N$ to the circle $|z|=r$ (figure 6). As soon as $s\not=r$, this does not a priori contradict the non conjugacy of~$F_r$ to a rotation.
\smallskip

\noindent Nevertheless, divergence may occur as is shown by the following Theorems {\sl in which contraction could be present but does not play any part in the proofs.}
\smallskip

\subsubsection{A criterion of divergence}
\noindent {\it In this section and the following, we choose to consider real analytic maps like $F(z)$ as series in two variables $z$ and $\bar z$ and hence we change the notation, writing $F(z,\bar z)$. We shall then note 
$F^0(z)=F(z,0)$ its holomorphic part. }

\begin{theorem} \label{hol1} Let $F(z,\bar z)=\lambda z(1+f(|z|^2))e^{2\pi ig(z,\bar z)}$ be a local analytic diffeomorphism from $(\R^2,0)$ to $(\R^2,0)$ such that the complex holomorphic map in one variable,  
$F^0(z)=\lambda ze^{2\pi ig(z,0)}$ be analytically non linearizable. Then any formal conjugacy $\Psi$ of $F(z,\bar z)$ to a normal form $N(z,\bar z)$ is divergent.
\end{theorem}

\begin{proof} 
From section \ref{Nonu} we know that the most general conjugacy $\Psi$ and normal form $N$ have the form
\begin{equation*}
\begin{split}
\Psi(z,\bar z)&=z(1+a(|z|^2))e^{2\pi i\varphi(z,\bar z)+b(|z|^2)},\\
N(z,\bar z)&=\lambda z(1+\alpha(|z|^2))e^{2\pi i\beta(|z|^2)}\; ,
\end{split}
\end{equation*}
hence
$\Psi^0(z)=ze^{2\pi i\varphi(z,0)}$ and $N^0(z)=\lambda z\; .$ 
\smallskip

\noindent The proof consists in the following identities, true for the maps we consider but certainly not for general maps\footnote{Thanks to Abed Bounemoura for insisting on this.}: 
$$(\Psi\circ F)^0=\Psi^0\circ F^0\quad \hbox{and}\quad (N\circ \Psi)^0=N^0\circ \Psi^0.\eqno(*)$$
Indeed, as $\Psi\circ F=N\circ \Psi$, this implies that $\Psi^0\circ F^0=N^0\circ \Psi^0$; 
convergence of $\Psi(z,\bar z)$ (and hence $N(z,\bar z)$) implying\footnote{Consider $\Psi$ as a function of two independent variables $z$ and $\bar z$.} that of $\Psi^0$, we would conclude to the analytic linearizability of $F^0$, a contradiction. 
\smallskip

\noindent The proof of $(*)$ consists in the following explicit computations which use the fact that $g(z,\bar z)$ and $\varphi(z,\bar z)$ are both real valued: 
\begin{equation*}
\begin{split}
(N\circ \Psi)(z,\bar z)&=\lambda z\bigl(1+a(|z|^2)\bigr)e^{2\pi i\bigl(\varphi(z,\bar z)+b(|z|^2)\bigr)}\bigl(1+\alpha(|\Psi(z,\bar z)|^2)\bigr)e^{2\pi i\beta(|\Psi(z,\bar z)|^2)},\\
(\Psi\circ F)(z,\bar z)&=\lambda z\bigl(1+f(|z|^2)\bigr)e^{2\pi ig(z,\bar z)}\bigl(1+a(|F(z,\bar z)|^2)\bigr)e^{2\pi i\bigl(\varphi(F(z,\bar z),\bar F(z,\bar z))+b(|F(z,\bar z)|^2)\bigr)}\, . 
\end{split}
\end{equation*}
Hence
\begin{equation*}
(N\circ \Psi)^0(z)=\lambda ze^{2\pi i\varphi(z,0)},\quad 
(\Psi\circ F)^0(z)=\lambda ze^{2\pi ig(z,0)}e^{2\pi i\varphi(F,\bar F)^0(z)}\, , 
\end{equation*}
while
\begin{equation*}
(N^0\circ \Psi^0)(z)=\lambda ze^{2\pi i\varphi(z,0)},\quad 
(\Psi^0\circ F^0)(z)=\lambda ze^{2\pi ig(z,0)}e^{2\pi i\varphi(F^0(z),0)}\, . 
\end{equation*}
It only remains to prove that $\varphi(F,\bar F)^0=\varphi(F^0,0)$: 
\smallskip

\noindent If $\varphi(z,\bar z)=\sum_{jk}c_{jk}z^j\bar z^k$, 
$$\varphi(F,\bar F)(z,\bar z)=\sum_{jk}c_{jk}\lambda^j{\bar\lambda}^kz^j{\bar z}^k\bigl(1+f(|z|^2)\bigr)^{j+k}e^{2\pi i(j-k)g(z,\bar z)},$$
hence $\; \varphi(F,\bar F)^0(z)=\sum_jc_{j_0}\lambda^jz^je^{2\pi ijg(z,0)}=\varphi(F^0(z),0)$.

\end{proof}

\begin{corollary}\label{A-div}
If $\omega$ is not a Brjuno number, any formal conjugacy of $A_{e^{2\pi i\omega},a,d}$ to a normal form diverges.
\end{corollary}

\begin{proof}
  We have $A_{\lambda,a,d}^0(z)=\lambda ze^{\pi z}$ with
  $\lambda=e^{2\pi i\omega}$, $\omega$ an irrational number.  Yoccoz
  had proved (see \cite{Y,PM}) that if one replaces
  $F^0(z)=\lambda ze^{\pi z}$ by the beginning $z(1+\pi z)$ of its
  Taylor expansion, the linearization converges if and only if
  $\omega$ is a Brjuno number. It was later on proved by Lukas Geyer
  (see \cite{G})\footnote{Thanks to Ricardo P\'erez-Marco for this
    reference} that the same is true for $F^0(z)$.
\end{proof}

\medskip

\noindent {\bf Remark.} As $B_{\lambda,a,d}^0(z)=\lambda z$, the above result does not apply to $B_{\lambda,a,d}$.
Notice that the sub-family of the Arnold family entering in the definition of $B_{\lambda,a,d}$ is much closer to a family of rotations than the one entering in the definition of $A_{\lambda,a,d}^0=\lambda z$.  Nevertheless the following strenghtening of Theorem \ref{hol1} allows concluding also for the maps $B_{\lambda,a,d}$.
\subsubsection{A more refined criterion of divergence}
\noindent Recall the homological equation for a special normalization $\Phi(z)=ze^{2\pi i\varphi(z)}$
which conjugates the local (formal) diffeomorphism $F(z)=\lambda z\bigl(1+f(|z|^2)\bigr)e^{2\pi ig(z)}$ to the normal form $N(z)=\lambda z\bigl(1+f(|z|^2)\bigr)e^{2\pi in(|z|^2)}$, i.e.\ such that  $\Phi\circ F=N\circ \Phi$~:
\begin{equation*}
\begin{split}
&g(z)-n(|z|^2)+\varphi\circ F(z)-\varphi(z)=0, \quad\hbox{that is}\\
&\sum_{j+k\ge 1}g_{jk}z^j\bar z^k-\sum_{s\ge 1}n_s|z|^{2s}\\
+&\sum_{j+k\ge 1}\varphi_{jk}\left[\lambda^j\bar\lambda^k\bigl(1+f(|z|^2)\bigr)^{j+k}e^{2\pi i(j-k)g(z)}-1\right]z^j\bar z^k
=0\, .
\end{split}
\end{equation*}
This implies that, if $p\not=q$, the coefficient $\varphi_{pq}$ satisfies
\begin{equation*}
(\lambda^p\bar\lambda^q-1)\varphi_{pq}=-g_{pq}+R,
\end{equation*}
where $R$ is the sum of all coefficients of $z^p\bar z^q$ in the expression
$$\sum_{1\le j+k<p+q}\varphi_{jk}\left[\lambda^j\bar\lambda^k\bigl(1+f(|z|^2)\bigr)^{j+k}e^{2\pi i(j-k)g(z)}-1\right]z^j\bar z^k\, .
$$
\begin{lemma}
Suppose there exists $\rho\le 1$ such that $g_{pq}\not=0$ implies $p-q\le \rho(p+q)$. Then the same is true for the coefficients $\varphi_{pq}$ of $\varphi$.
\end{lemma}
\begin{proof}
Of course the lemma is empty if $\rho=1$.
We suppose by induction that for any couple $(j,k)$ such that $1\le j+k<p+q$, 
$\varphi_{jk}\not=0$ implies $(j-k)\le \rho(j+k)$. The formula above shows that the property is still true for $\varphi_{pq}$. Indeed, each term of $R$ is a product of terms $A_iz^{p_i}\bar z^{q_i}$ each of which satisfies $p_i-q_i\le \rho(p_i+q_i)$.
\end{proof}
\medskip

\noindent{\it Notations.} 
Let $\rho=\frac{N}{M}=\sup_{g_{pq}\not=0}\frac{p-q}{p+q}$. Supposing 
$\frac{N}{M}$ irreducible, and noting $z=re^{2\pi it}$, we define 
$$Z=r^{M}e^{2\pi iNt}$$ 
and consider the pairs $(p_k,q_k)$ of non negative integers such that  
$$m_k=p_k+q_k=kM,\quad n_k=p_k-q_k=kN.$$
Notice that such pairs need not exist for all $k\ge 1$: for example, if $M=2,N=1$, $(p_1,q_1)$ is not a pair of integers. 
\noindent Let $F^0(Z)$ and $\Phi^0(Z)$ be defined by
$$F^0(Z)=\lambda^NZe^{2\pi iNg^0(Z)},\quad \Phi^0(Z)=Ze^{2\pi iN\varphi^0(Z)},$$
where
$$g^0(Z)=\sum_{k}g_{p_kq_k}z^{p_k}\bar z^{q_k}=\sum_{k}g_{p_kq_k}Z^k,\quad 
\varphi^0(Z)=\sum_{k}\varphi_{p_kq_k}z^{p_k}\bar z^{q_k}=\sum_{k}\varphi_{p_kq_k}Z^k,$$
the sums being taken over the set of integers $k$ such that $(p_k,q_k)$ is defined.
\begin{lemma}\label{lin} 
If $\rho=\frac{N}{M}$, the homological equation implies
$$g^0(Z)+\varphi^0\circ F^0(Z)-\varphi^0(Z)=0.$$
In other words, $\Phi^0(Z)$ linearizes $F^0(Z)$:
$$\Phi^0\circ F^0=L\circ \Phi^0,$$
where $L(Z)=\lambda^N Z.$ 
\end{lemma}

\begin{proof}
Developing the homological equation we get
\begin{multline*}
  \hspace{-1.3em}
  \sum_{j,k}g_{jk}z^j\bar z^k-\sum_sn_s|z|^{2s}\\[.7ex]
  + \sum_{j,k}\varphi_{jk}
\bigg[ \lambda^j\bar\lambda^k \Big(1+\sum_uf_u|z|^{2u}\Big)^{j+k}
\sum_n\frac{(2\pi i(j-k))^n}{n!}
\Big(\sum_{v,w}g_{vw}z^v\bar z^w\Big)^n -1\bigg]
z^j\bar z^k \\[.7ex]
 =0. \hspace{-2em}
\end{multline*}
The general term $z^p\bar z^q$ in the last line has the form
$$p=j+u_1+u_2+\cdots+u_{j+k}+v_1+v_2+\cdots +v_n,\quad q=k+u_1+u_2+\cdots+u_{j+k}+w_1+w_2\cdots+w_n.$$
As 
$j-k\le \rho(j+k)\quad \hbox{and}\quad \forall i, v_i-w_i\le \rho(v_i+w_i),$
the only possiblity for achieving $p-q=\rho(p+q)$ is
$$j-k=\rho(j+k)\quad, \forall i, \, u_i=0\quad  \hbox{and}\quad \forall j,\,  v_j-w_j=\rho(v_j+w_j).$$ 
Hence, restricting the summations to those pairs of indices which satisfy the above identities $j-k=\rho(j+k)$ and $v-w=\rho(v+w)$, 
we get
\begin{equation*}
\sum_{j,k}g_{jk}z^j\bar z^k
+\sum_{j,k}\varphi_{jk}\left[\lambda^j\bar\lambda^k\sum_n\frac{(2\pi i(j-k))^n}{n!}\left(\sum_{v,w}g_{vw}z^v\bar z^w\right)^n-1\right]z^j\bar z^k=0,
\end{equation*}
that is
\begin{equation*}
\sum_{l\ge 1}g_{p_lq_l}Z^l+\sum_{l\ge 1}\varphi_{p_lq_l}\left[\lambda^{lN}\exp\left(2\pi ilN\sum_{s\ge 1}g_{p_sq_s}Z^s\right)-1\right]Z^l=0,
\end{equation*}
or
$$g^0(Z)+\varphi^0\left(\lambda^NZ\exp\left(2\pi iNg^0(Z)\right)\right)-\varphi^0(Z)=0,$$
which is equivalent to the linearization equation
$$\Phi^0\circ F^0(Z)=\lambda^N\Phi^0(Z).$$
\end{proof} 
\begin{theorem}\label{hold}
Under the hypotheses of lemma \ref{lin}, if the the holomorphic germ $F^0(Z)=\lambda^Ne^{2\pi iNg^0(Z)}$ is not holomorphically linearizable, any formal conjugacy $\Psi$ of the germ $F(z)=\lambda z\bigl(1+f(|z|^2)\bigr)e^{2\pi ig(z)}$ to a normal form is divergent.
\end{theorem}
\begin{proof}
Lemma \ref{lin} is still valid if $\Phi$ is replaced by any formal conjugacy $\Psi$ of $F$ to a normal form.
Indeed, replacing $\Phi(z)$ by $$\Psi(z)=H\circ\Phi(z)=z\left(1+a(|z|^2)\right)e^{2\pi i\left(\varphi(z)+b(|z|^2)\right)}$$
does not change the proof because monomials of the form $|z|^{2s}$ never participate in the ones $z^p\bar z^q$ achieving the maximum of
$\frac{p-q}{p+q}\,\cdot$ 
\end{proof}
\begin{corollary}\label{divB} If $\omega$ is not a Brjuno number, any formal conjugacy of $B_{e^{2\pi i\omega},a,d}$ to a normal form diverges.
\end{corollary}
\begin{proof}
We have  $\rho=1/3, Z=z^2\bar z$ and $B_{e^{2\pi i\omega},a,d}^0(Z)=Ze^{\pi Z}$.
\end{proof}
\medskip

\noindent {\bf A question.} Here is a simple example for which Theorem \ref{hold} does not lead to a conclusion and hence leaves unsettled the question of divergence:
$$C_{\lambda,a,d}(z)= \lambda z(1+a|z|^{2d})e^{2\pi i|z|^2(1+\Im e^z)}.$$
Indeed,
$$\sup_{g_{pq}\not=0}\frac{p-q}{p+q}=\sup_n\frac{n}{n+2}=1.$$

\noindent Hence only Theorem 10 applies, but $F^0(z)=\lambda z$.
 \subsection {The case of strong contraction $|\lambda|<1$}\label{strong}

\noindent If $\rho=|\lambda|\not=$1, Poincar\'e's theorem insures the existence of an analytic local conjugacy of $F$ to its derivative $dF(0)$ but also to any convergent normal form $N(z)=\lambda z(1+\sum_{k\ge 1}\alpha_k|z|^{2k})$ ($\alpha_k\in\C$). The difference with the case $|\lambda|=1$ is the possibility of fixing arbitrarily the series $n(|z|^2)$ by choosing the coefficients $\varphi_{pp}$ of a conjugacy $z\mapsto \Phi(z)=ze^{2\pi i\varphi(z)}$. But there is a unique 
formal diffeo\-morphism tangent to Identity which conjugates $F$ to its derivative $dF(0)$: indeed, if $\Psi$ is another one, the composition $h=\Psi\circ \Phi^{-1}$ is tangent to Identity and satisfies the equation $h(\lambda z)=\lambda h(z)$; a term by term identification of the series expansion of $h$ then shows that, already at the formal level, $h$ is the Identity. Hence, if $|\lambda|<1$ the analytic linearization $\Phi$ of $A_{\lambda}$ is of the form $\Phi(z)=ze^{2\pi i\varphi(z)}$ where $\varphi$ is the convergent solution of the equations 
\begin{equation*}
\left\{
\begin{split}
&\varphi_{10}=\frac{1}{2i(1-\lambda)},\quad\quad \varphi_{01}=\frac{1}{2i(\bar\lambda-1)}=\bar\varphi_{10},\\
&\sum_{j+k\ge 2}\varphi_{jk}(\lambda^j\bar\lambda^k-1)z^j\bar z^k+\sum_{j+k\ge 1,n\ge 1}\varphi_{jk}\lambda^j\bar\lambda^k\frac{\pi^n}{n!}
(j-k)^n(z-\bar z)^nz^j\bar z^k=0.
\end{split}
\right.
\end{equation*}
Notice that in this case the formula 
$\varphi(z)=\sum_{m=0}^\infty\tilde g\circ F^{(m)}(z)$
which one deduces immediately by iterating the homological equation
makes sense in the realm of power series while it makes sense only for each fixed $z$ in the case of weak contraction (Lemma \ref{ptwisecv}).

\subsection{Always convergence or generic divergence}
In \cite{PM2}, Ricardo P\'erez-Marco showed that, for the Birkhoff normal form of an analytic Hamiltonian flow at a non-resonant
singular point with given quadratic part, as well as for the
normalizing transformation, the following alternative holds: either it
is always convergent or it is generically divergent. We now show how
to adapt the proof to the non conservative case in our context.
\smallskip

\noindent Let $\lambda=e^{2\pi i\omega}$ with real $\omega\not\in\Q$. Consider the following families of local real analytic diffeomorphisms of~$\R^2$
(the lower indices indicate the elements which are fixed in the family):

\begin{equation*}
\begin{split}
{\cal F}_{\lambda}&=\left\{F\left|  F(z)=\lambda z(1+f(|z|^2)e^{2\pi i
      g(z)},\; \hbox{$f,g$ arbitrary}\right.\right\},\\
{\cal F}_{\lambda,\bull,g}&=\left\{F\left|  F(z)=\lambda z(1+f(|z|^2)e^{2\pi
      i g(z)},\; \hbox{$f$ arbitrary}\right.\right\},\\
{\cal F}_{\lambda,f,\bull}&=\left\{F\left|  F(z)=\lambda z(1+f(|z|^2)e^{2\pi
     i g(z)},\; \hbox{$g$ arbitrary}\right.\right\}.
\end{split}
\end{equation*}
with real analytic functions
\begin{equation} \label{eqrealfg}
f(u)=\sum_{j\geq 1}f_j u^j\enspace (f_j\in \R)
\enspace\text{and}\enspace g(z)=\sum_{j+k\geq 1}g_{jk}z^j\bar{z}_k
\enspace (g_{jk}=\bar{g}_{kj}).
\end{equation}
At the end of this section, we will prove

\begin{theorem}\label{cases-total-divergence}
  Let $\omega\in\R\setminus\Q$.
  \smallskip
  
  \noindent i) Let~$g$ be as in~\eqref{eqrealfg}. The generic element
  of ${\cal F}={\cal F}_{e^{2\pi i\omega}}$ or
  ${\cal F}={\cal F}_{e^{2\pi i\omega},\bull,g}$ has no convergent
  normalization.   \smallskip
  
  \noindent ii) Let $f\not\equiv 0$ be as in~\eqref{eqrealfg}. If
  $\omega$ is not a Brjuno number, 
 then the generic element of
  ${\cal F}={\cal F}_{e^{2\pi i\omega},f,\bull}$ has no convergent
  normalization.
\end{theorem}

\noindent Recall that a normalization $\Phi^*$ of $F$ is called the  basic normalization if 
\[\Phi^*(z)=ze^{2\pi i \varphi^*(z)}\quad\text{with}\enspace \varphi^*(z)=\sum_{p+q\geq
    1}\varphi_{pq}^*z^p\bar{z}^q \enspace
 \text{where} \enspace \varphi_{pp}^*=0\enspace\text{for}\enspace p\geq
  1,\]
and the corresponding normal form is called the basic normal form:
\[N^*(z)=\lambda z(1+f(|z|^2))e^{2\pi i n^*(|z|^2)},\quad n^*(z)=\sum_{s\geq 1} n_s^*|z|^{2s}.\]

\noindent The main part of this section is devoted to the proof of

\begin{theorem}\label{generic-div} Let $\lambda=e^{2\pi i\omega}$ with
  $\omega \in \R\setminus\Q$ and let
  ${\cal F}$ be one of the three families of local real analytic diffeomorphisms defined above.
Either the basic normalization of every $F\in{\cal F}$ is convergent
(resp. its basic normal form is convergent), or the normalizations of
a generic $F\in{\cal F}$ are divergent (resp. its basic normal form is
divergent).
\end{theorem}

\noindent The proof will require three lemmas.

\begin{lemma}\label{deg}
Let $F(t;z)=\lambda z(1+f(t,|z|^2))e^{2\pi i g(t,z)}$ be a family of local maps where
\[
 f(t,u)=\sum_{j\geq 1}f_j(t)u^j,\; f_j(t)=\overline{f_j(\bar t)}, \quad
g(t,z)=\sum_{j+k\geq
    1}g_{jk}(t)z^j\bar{z}_k, \; g_{jk}(t)=\overline{g_{kj}(\bar t)},
\]
and the coefficients $f_j(t)$ and $g_{jk}(t)$ are polynomial functions
of $t\in\C$.
Then the basic normalization $\Phi^*(t;z)$ has the property that each
$\varphi_{pq}^*(t)$ is polynomial in~$t$ with degree no larger than
$p+q$, and the basic normal form $n^*(t;z)$ has its coefficients
$n_s^*(t)$ polynomial in~$t$ with degree no larger than~$2s$.
\end{lemma}

\begin{proof}   
Recall that the
conjugacy equation $\Phi^*\circ F=N^*\circ \Phi^*$ is reduced to the homological equation
$$g(t;z)-n^*(t,|z|^2)+\varphi^*(t)\circ F(t;z)-\varphi^*(t;z)=0.$$
Writing
$$(1+f(t;|z|^2))e^{2\pi i g(t;z)}=1+\sum_{\alpha+\beta\geq 1} c_{\alpha\beta}(t)z^{\alpha}\bar{z}^{\beta},$$
since $f_j(t)$ and $g_{pq}(t)$ are polynomial in~$t$, the coefficient $c_{\alpha\beta}(t)$ is a polynomial in $t$ of degree no larger than $\ell=\alpha+\beta$.
\smallskip

\noindent The conjugacy equation becomes
\begin{eqnarray*}&&-\sum_{j+k\geq 1}g_{jk}(t)z^j\bar{z}^k+\sum_{s\geq 1}n^*_{s}(t)|z|^{2s}\\
&&=\sum_{p+q\geq 1}\varphi_{pq}^*(t)z^p\bar{z}^q[\lambda^p\bar{\lambda}^q(1+\sum_{\alpha+\beta\geq 1} c_{\alpha\beta}(t)z^{\alpha}\bar{z}^{\beta})^p(1+\sum_{\alpha+\beta\geq 1}\overline{ c_{\alpha\beta}}(t)z^{\beta}\bar{z}^{\alpha})^q-1]\\
&&=\sum_{p+q\geq 1}\varphi_{pq}^*(t)z^p\bar{z}^q(\lambda^p\bar{\lambda}^q-1)+A_{pq}(t)z^p\bar{z}^q,
\end{eqnarray*}
where $A_{pq}(t)$ given by summation and multiplication of
$\varphi_{jk}^*(t)$ and $c_{jk}^*(t)$ with $j+k<p+q$, whence
$A_{pq}(t)$ is a polynomial function of $t$ .\smallskip

\noindent The coefficients $\varphi_{pq}^*(t)$ and $n^*_{s}(t)$ are
uniquely determined by induction on the degree $\ell:=p+q$, once the
$\varphi_{pp}^*$ are chosen to be zero. By induction, we get
$A_{pq}(t)$ of degree smaller than $p+q$, $\varphi^*_{pq}(t)$
polynomial function of~$t$ with degree no larger than $p+q$,
and $n_s^*(t)$ polynomial with degree no larger than $2s$.
\end{proof}

\medskip

\noindent In the following, we will be using the notion of a polar set,
the Green function~$V_E$ of a subset~$E$ of~$\C$, and the
Bernstein-Walsh lemma; the reader is referred to P\'erez-Marco's
paper \cite{PM2} for all this.
\smallskip

\noindent Let $(f_0,g_0)$ and  $(f_1,g_1)$ be as
in~\eqref{eqrealfg}. We will consider the affine subspace~$V$
consisting of the maps
\[
F_t(z) = \lambda z(1+(tf_0+(1-t)f_1)(|z|^2))e^{2\pi i
  (tg_0(z)+(1-t)g_1(z))},
\qquad t\in \C.
\]

\begin{lemma} Let~$E$ denote the set of parameters $t \in \C$ such that the basic
  normalization $\Phi_t^*$ (resp. the basic normal form $N^*_t$) is
  convergent.
If~$E$ is not polar, then $E=\C$.
\end{lemma}

\begin{proof} Let~$E$ denote the set of parameters $t \in \C$ such
  that~$\Phi_t^*$ is convergent and suppose that~$E$ is not polar.
We have
$$E=\cup_{n\geq 1}E_n,$$
where $E_n$ is the set of $t\in E$ such that the power series
$\varphi^*_t(z)$ is convergent and bounded by $1$ for $|z|\le 1/n$.
We can thus find $n\ge1$ such that~$E_n$ is not polar.
According to Lemma \ref{deg}, we have
$$\varphi_t^*(z)=\sum_{j+k\geq 1}\varphi_{jk}^*(t)z^j\bar{z}^k$$
where $\varphi_{jk}^*(t)$ depends polynomially on~$t$ with degree no
larger than $j+k$.
The Cauchy inequalities (viewing $\varphi^*_t$ as a function of two
independent variables $(z,\bar{z})$) yield
\[|\varphi_{jk}^*(t)|\leq n^{j+k}. \] 
By the Bernstein-Walsh lemma, we get that if $K\subset \C$ is compact
and $j+k\geq 2$, then
\[
\max_{t\in K}\|\varphi_{jk}^*(t)\| \leq C^{j+k} n^{j+k},
\qquad \text{where}\quad 
C =\exp\Big( \max_{t\in K}V_{E_n}(t) \Big).
\]
Hence $\varphi_t^*(z)$ is convergent for any $t\in \C$.
The argument for the set of parameters~$t$ such that $N^*_t$ is
convergent is similar.
\end{proof}

\begin{lemma}\label{generic1} If there exists $t\in\C$ such that
  $\Phi^*_t$ (resp. $N^*_t$) is divergent, then the set of parameters
  $t\in \C$ (resp. $t\in \R$) with convergent normalization~$\Phi^*_t$
  (resp. basic normal form~$N^*_T$) has Lebesgue measure zero.
\end{lemma}

\begin{proof} It follows from the fact that a polar subset of~$\C$ is of
  Lebesgue measure zero, and the intersection of a polar subset of~$\C$
  with~$\R$ is of Lebesgue measure zero.
\end{proof}
\medskip

\noindent {\bf Proof of Theorem \ref{generic-div}}:
Let ${\cal F} = {\cal F}_\lambda$.
Suppose that there exists $F_0\in{\cal F}$ the basic normalization of
which is divergent.  For $n\ge1$, denote by $E_n\subset \mathcal{F}$ the
set of $F\in \mathcal{F}$ which have convergent basic
normalization with $\varphi^*(z)$ convergent and
bounded by $1$ for $|z|\le 1/n$. It is easy to check that
each~$E_n$ is a closed set. Now
$$E=\cup_{n\geq 1}E_n$$
is the set of $F$ in $\mathcal{F}$ having a convergent basic
normalization.

\smallskip

\noindent Let $n\ge1$. We claim that the set $\mathcal{F}-E_n$ is dense. Otherwise,
there exists a map $F_1$ in the interior of~$E_n$.
Consider the subspace
\begin{multline*}
 V = \{\, F_t \mid F_t(z)=\lambda z(1+(tf_0+(1-t)f_1)(|z|^2))e^{2\pi i
    (tg_0(z)+(1-t)g_1(z))},\\ t\in \C \;(\text{resp.}\; \R) \,\}
\end{multline*}
By Lemma~\ref{generic1}, the set of parameters~$t$ giving rise to a convergent basic normalization has measure zero. But on the other hand it contains a neighborhood of~$0$, contradiction.
Therefore, the set of maps~$F$ in $\mathcal{F}$ with divergent basic normalization 
$$\mathcal{F}-E=\bigcap_{n\geq 1}(\mathcal{F}-E_n)$$
is a countable intersection of open dense set.

\smallskip

\noindent
Finally, recall that by Lemma~\ref{normalization}, if the basic
normalization of a map $F\in {\cal F}$ is divergent, then
all normalizations of~$F$ are divergent.

\smallskip

\noindent An analogous argument works for the families 
$\mathcal{F}_{\lambda,f,\bull}$  and $\mathcal{F}_{\lambda,\bull,g}$, by taking $g_0=g_1$ and $f_0=f_1$ respectively, which ends the proof of Theorem~\ref{generic-div}.

\medskip

\noindent {\bf Proof of Theorem \ref{cases-total-divergence}}:
Theorem~\ref{generic-div} gives an alternative: total convergence of
the basic normalization or
generic divergence of the normalizations.
In case~(i) it's generic divergence, in view of the existence of divergent conservative examples.
In case~(ii) too, in view of Corollary \ref{A-div} (or, more
accurately, its analogue where we replace $a|z|^{2d}$ with an
arbitrary $f(|z|^2)\not=0$).

\subsection{More questions.} 

\noindent {\bf 1) Nature of the special normal forms in the conservative case} 
By section \ref{pol}, polynomial normal forms always exist. As, in the conservative case we know that they correspond to non conservative conjugacies, this leaves open the question of the nature of the special normal forms, namely: is divergence of special normal forms $N$ generic in the conservative case? Recall that we know that the conjugacy itself to the special normal form is in general divergent.
\smallskip

\noindent {\bf 2) Nature of the special normal forms and special conjugacies in case of weak contraction} 
Is divergence of the special normal form, the conjugacy $\Phi$ and more generally of any conjugacy $\Psi$ to a normal form generic when 
$\omega$ is not a Brjuno number?  
\smallskip

\goodbreak

\noindent {\bf 3) What about the role of translated objects?}
In this case, there are no more invariant objects but only translated objects. Indeed, in the simple case that we are considering, the circle of radius $r$ centered at 0 is radially {\it translated} by $F$ onto the circle centered at 0 of strictly smaller radius $s=r(1+f(r^2))$. Let us call $\rho(r)\in\R/\T$ the rotation number of the diffeomorphism $g_r$ of $\R/\T$ defined by the restriction to the circle of radius $r$ of the argument $2\pi g$ of $F$. The values of $r$ such that $\rho(r)=p/q\in\Q/\T$ define {\it resonant annuli}. One can show that inside each such annulus there is a curve of {translated periodic orbits} of rotation number $p/q$, the translation depending on the orbit\footnote{A translated orbit is an orbit whose image under $F$ is obtained by a radial translation by some constant. They exist independently of the hypothesis that $F$ preserves the foliation ${\cal F}_0$.}

\hskip +0.3cm
\includegraphics[scale=0.35]{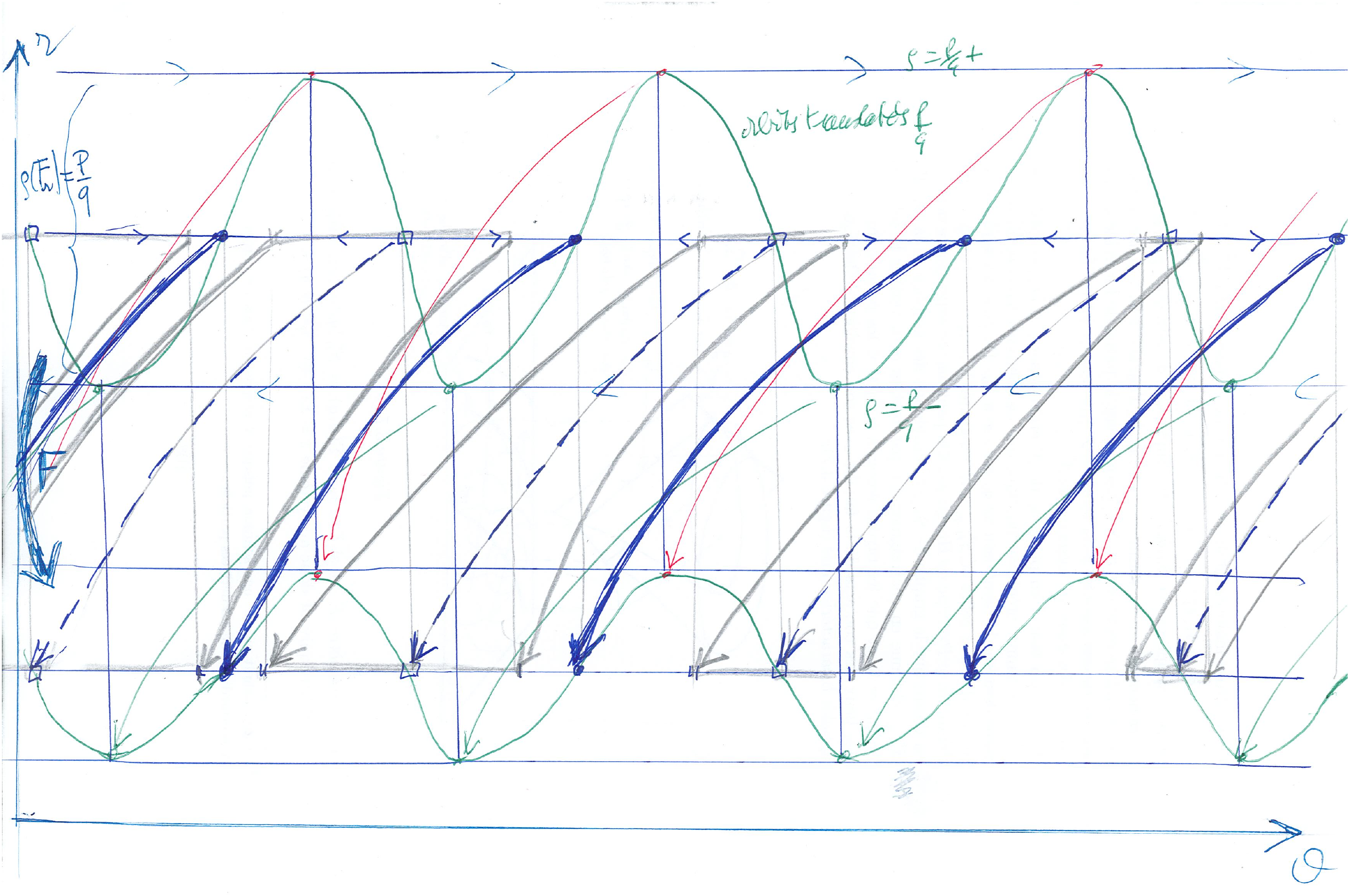}
\begin{center}
\vskip-0.5cm 
Figure 6 : Resonant zones.
\end{center}

\noindent Is the relation between the strength of attraction and the measure, in some system of local coordinates, of the set of translated circles whose rotation number is rational (the resonant zones) relevant to the conjugacy problem? In particular, are the results for 
$A_{\lambda,a,d}$ and $B_{\lambda,a,d}$ different?
\smallskip

\noindent {\it The problem is, of course, that translated objects are not invariant under conjugacy} ; in particular, in the case of a strong contraction the existence in some local coordinates of resonant zones, does not prevent analytical conjugacy to a rotation (see section
\ref{strong})!
\medskip

\subsection*{Thanks} to Jacques F\'ejoz, Abed Bounemoura and Ricardo P\'erez-Marco for fruitful questions and discussions.
\smallskip

\noindent
The first two authors thank Capital Normal University for its
hospitality.
\smallskip

\noindent
The third author is partially supported by National Key R\&D Program
of China \\ (2020YFA0713300), NSFC (No.s 11771303, 12171327, 11911530092,
11871045).



\end{document}